\documentclass[12pt,reqno]{amsart}
\usepackage[margin=1in]{geometry}
\usepackage{amsmath,amssymb,amsthm,graphicx,amsxtra, setspace}
\usepackage[utf8]{inputenc}
\usepackage{mathrsfs}
\usepackage{eucal}
\usepackage{hyperref}
\usepackage{upgreek}
\usepackage{mathtools}
\usepackage{multirow}
\usepackage{tikz}
\usepackage{mathabx}
\usepackage{graphicx,type1cm,eso-pic,color}
\allowdisplaybreaks

\usepackage[pagewise]{lineno}

\DeclareMathAlphabet{\mathpzc}{OT1}{pzc}{m}{it}

\usepackage[cyr]{aeguill}

\colorlet{darkblue}{blue!50!black}

\hypersetup{
	colorlinks,%
	citecolor=blue,%
	filecolor=red,%
	linkcolor=red,%
	urlcolor=blue,%
	pdfnewwindow=true,%
	pdfstartview={FitH}
}

\newtheorem{theorem}{Theorem}[section]
\newtheorem{lemma}[theorem]{Lemma}
\newtheorem{proposition}[theorem]{Proposition}

\newtheorem{definition}[theorem]{Definition}
\newtheorem{problem}[theorem]{Problem}
\newtheorem{algorithm}[theorem]{Algorithm}

\newtheorem{remark}[theorem]{Remark}

\newtheorem{hypothesis}[theorem]{Hypothesis}

\allowdisplaybreaks

\let\originalleft\left
\let\originalright\right
\renewcommand{\left}{\mathopen{}\mathclose\bgroup\originalleft}
\renewcommand{\right}{\aftergroup\egroup\originalright}

\renewcommand{\d}{\/\mathrm{d}\/}

\def\w{\textbf{W}^{\varepsilon}_{{\theta}^{\varepsilon}}}

\def\S{\mathrm{S}}

\def\L{\mathbb{L}}
\def\A{\mathcal{A}}

\def\C{\mathrm{C}}
\def\f{\boldsymbol{f}}

\def\B{\mathcal{B}}

\def\y{\boldsymbol{y}}

\def\X{\mathbb{X}}
\def\x{\boldsymbol{x}}

\def\z{\boldsymbol{z}}
\def\v{\boldsymbol{v}}
\def\V{\mathbb{v}}
\def\w{\boldsymbol{w}}

\def\N{\mathbb{N}}
\def\bu{\boldsymbol{u}}
\def\bv{\boldsymbol{v}}

\def\V{\mathbb{V}}
\def\wi{\widetilde}

\def\bu{\mathrm{U}}

\def\bu{\boldsymbol{u}}
\def\H{\mathbb{H}}
\def\n{\boldsymbol{n}}

\newcommand{\R}{\mathbb{R}}

\renewcommand{\d}{\/\mathrm{d}\/}

\newcommand{\Addresses}{{
		\footnote{
			
			\noindent \textsuperscript{1}Department of Mathematics, Indian Institute of Technology Roorkee-IIT Roorkee,
			Haridwar Highway, Roorkee, Uttarakhand 247667, INDIA.\par\nopagebreak
			\noindent  \textit{e-mail:} \texttt{Manil T. Mohan: maniltmohan@ma.iitr.ac.in, maniltmohan@gmail.com.}
			
			\noindent \textsuperscript{*}Corresponding author.
			
			\textit{Key words:} Convective Brinkman-Forchheimer extended Darcy equations, Hemivariational inequality, Banach contraction mapping principle, Schauder’s fixed point theorem, Convex minimization. 
			
			Mathematics Subject Classification (2020): Primary  47H10; 49J40. Secondary 35Q35, 76D03, 49K40.

}}}

\begin{document}
	
	\title[Stationary  2D and 3D CBFeD  Hemivariational inequalities]{Well-posedness of  stationary 2D and 3D convective Brinkman-Forchheimer extended Darcy Hemivariational inequalities
		\Addresses}
	\author[M. T. Mohan]
	{Manil T. Mohan\textsuperscript{1*}}
	
	\maketitle

	\begin{abstract}
		This study addresses the well-posedness of a hemivariational inequality derived from the convective Brinkman-Forchheimer extended Darcy (CBFeD) model in both two and three dimensions. The CBFeD model describes the behavior of incompressible viscous fluid flow through a porous medium, incorporating the effects of convection, damping, and nonlinear resistance. The mathematical framework captures steady-state flow conditions under a no-slip boundary assumption, with a non-monotone boundary condition that links the total fluid pressure and the velocity’s normal component through a Clarke subdifferential formulation. To facilitate the analysis, we introduce an auxiliary hemivariational inequality resembling a nonlinear Stokes-type problem with damping and pumping terms, which  serves as a foundational tool in establishing the existence and uniqueness of weak solutions for the CBFeD model. The analytical strategy integrates techniques from convex minimization theory with fixed-point methods, specifically employing either the Banach contraction mapping principle or Schauder’s fixed point theorem. The Banach-based approach, in particular, leads to a practical iterative algorithm that solves the original nonlinear hemivariational inequality by sequentially solving  Stokes-type problems, ensuring convergence of the solution sequence.	Additionally, we derive equivalent variational formulations in terms of minimization problems. These formulations lay the groundwork for the design of efficient and stable numerical schemes tailored to simulate flows governed by the CBFeD model.
	\end{abstract}

	\section{Introduction}\label{sec1}\setcounter{equation}{0}

		\subsection{The model}\label{sub-boundary}
Let $\mathcal{O} \subset \mathbb{R}^d$ (with $d = 2$ or $3$) be a bounded, connected, open set whose boundary $\Gamma$ is Lipschitz continuous. A generic point in the domain $\mathcal{O}$ or on its boundary $\Gamma$ is denoted by $\boldsymbol{x} \in \mathbb{R}^d$. We begin with the formulation of the stationary two- and three-dimensional \emph{convective Brinkman-Forchheimer extended Darcy (CBFeD)} equations:
	\begin{equation}\label{eqn-stationary}
		\left\{
		\begin{aligned}
			-\mu \Delta\bu+(\bu\cdot\nabla)\bu+\alpha\bu+\beta|\bu|^{r-1}\bu+\kappa|\bu|^{q-1}\bu+\nabla p&=\boldsymbol{f}, \ \text{ in } \ \mathcal{O}, \\ \nabla\cdot\bu&=0, \ \text{ in } \ \mathcal{O},
		\end{aligned}
		\right.
	\end{equation}
where  $\bu(\cdot): \mathcal{O} \to \mathbb{R}^d$ denotes the velocity field, $p(\cdot): \mathcal{O} \to \mathbb{R}$ the pressure, and $\f(\cdot): \mathcal{O} \to \mathbb{R}^d$ the external force density. 
The coefficient $\mu > 0$ represents the \emph{Brinkman coefficient}, which accounts for the effective viscosity of the fluid. The positive constants $\alpha$ and $\beta$ correspond to the \emph{Darcy} and \emph{Forchheimer coefficients}, respectively, modeling resistance arising from the medium's permeability and porosity. A nonlinear term $\kappa|\bu|^{q-1}\bu$ is incorporated to represent a possible \emph{pumping mechanism}, particularly relevant when $\kappa < 0$, in which case it acts against the damping contributions $\alpha\bu + \beta|\bu|^{r-1}\bu$. Throughout this work, we assume $\kappa < 0$. The exponent $r \in [1, \infty)$, referred to as the \emph{absorption exponent}, characterizes the nonlinear damping. The case $r = 3$ is identified as the \emph{critical exponent} (see \cite[Proposition 1.1]{KWH}). The additional parameter $q$ satisfies $q \in [1, r)$, ensuring that the pumping term remains subcritical relative to the damping. Notably, the system reduces to the classical \emph{Navier–Stokes equations (NSE)} when $\alpha = \beta = \kappa = 0$. If $\alpha, \beta > 0$ and $\kappa = 0$, the model becomes a \emph{damped variant} of the NSE. The CBFeD model is derived from an extended Darcy-Forchheimer law, as discussed in \cite{SGKKMTM, MTT}. The second equation in \eqref{eqn-stationary} enforces the \emph{incompressibility condition} for the fluid flow.

Let us begin by rewriting the first equation in \eqref{eqn-stationary} using the curl operator. We utilize the following vector identities:
\begin{align}
	-\Delta\bu &= \mathrm{curl\ }\mathrm{curl\ }\bu - \nabla(\nabla \cdot \bu) = \mathrm{curl\ }\mathrm{curl\ }\bu, \label{eqn-vector-1} \\
	(\bu \cdot \nabla)\bu &= \mathrm{curl\ }\bu \times \bu + \frac{1}{2} \nabla|\bu|^2, \label{eqn-vector-2}
\end{align}
where we have used the fact that $\nabla \cdot \bu = 0$. Applying these identities, the system in \eqref{eqn-stationary} can be equivalently written as:
\begin{equation}\label{eqn-stationary-1}
	\left\{
		\begin{aligned}
			\mu\ \mathrm{curl\ }\mathrm{curl\ }\bu + \mathrm{curl\ }\bu \times \bu + \alpha\bu + \beta|\bu|^{r-1}\bu +\kappa |\bu|^{q-1}\bu + \nabla q &= \boldsymbol{f}, \quad \text{in } \mathcal{O}, \\
			\nabla \cdot \bu &= 0, \quad \text{in } \mathcal{O},
		\end{aligned}
		\right.
	\end{equation}
	where $q(\cdot) = p(\cdot) + \frac{1}{2}|\bu(\cdot)|^2$ denotes the dynamic pressure. The CBFeD equations given in \eqref{eqn-stationary-1} are complemented by appropriate boundary conditions. We now proceed to describe these associated boundary conditions.
Let $\boldsymbol{n} = (n_1, \ldots, n_d)^{\top}$ denote the unit outward normal vector on the boundary $\Gamma$. For a vector field $\bu$ defined on $\Gamma$, we define the normal and tangential components as $u_n = \bu \cdot \boldsymbol{n}$ and $\bu_{\tau} = \bu - u_n \boldsymbol{n}$, respectively. We impose the following boundary conditions:
\begin{align}
	\bu_{\tau} &= \boldsymbol{0} \quad \text{on } \Gamma, \label{eqn-bondary-1} \\
	q &\in \partial j(u_n) \quad \text{on } \Gamma, \label{eqn-bondary-2}
\end{align}
where $j(u_n)$ is shorthand for $j(\x, u_n(\x))$. The function $j : \Gamma \times \mathbb{R} \to \mathbb{R}$ is referred to as a \emph{superpotential} and is assumed to be locally Lipschitz continuous in its second variable. The symbol $\partial j$ denotes the \emph{Clarke subdifferential} of $j(\x, \cdot)$. The boundary condition \eqref{eqn-bondary-2} arises naturally in the context of fluid flow through a tube or channel, where the fluid entering $\mathcal{O}$ may exit through boundary orifices, whose sizes can be altered by an external device. In this setting, the normal component of the fluid velocity on the boundary is controlled to minimize the total pressure on $\Gamma$. Various boundary conditions correspond to different physical scenarios. It is important to note that if the function $j(\x, u_n)$ is convex with respect to its second argument, then the boundary conditions \eqref{eqn-bondary-1}-\eqref{eqn-bondary-2} lead to a \emph{variational inequality}. However, in our context, we do not assume convexity of $j$ with respect to its final argument. As a result, the problem under consideration constitutes a \emph{hemivariational inequality}.

	\subsection{Literature review}
	Hemivariational inequalities (HVIs) were first introduced by Panagiotopoulos in the early 1980s (\cite{PDP}) to address the challenges of modeling and analyzing engineering problems characterized by non-smooth and non-monotone interactions between physical quantities.   Foundational developments in the theory of HVIs are documented in several early monographs (see, for example, \cite{JHMMPD,ZNPDP,PDP1}). Over the past few decades, the modeling, theoretical analysis, numerical computation, and applications of HVIs have received increasing attention across various scientific fields. More recent progress in both the mathematical theory and engineering applications, particularly in contact mechanics, has led to the broader study of variational–hemivariational inequalities, with key results reported in \cite{SMAOMS,MSSM}. Advances in the numerical analysis of such inequalities have also been made, and a summary of current methodologies and results can be found in \cite{WHMS}.

	Hemivariational inequalities (HVIs) also emerge in various applications within fluid mechanics. The first analysis of a stationary Navier-Stokes HVI appeared in \cite{SMi,SMAO}, where the existence of solutions is established using a surjectivity theorem for a pseudomonotone and coercive operator ({\cite[Theorem 1.3.70]{ZdSmP}}). Subsequently, an evolutionary HVI was investigated in \cite{SMAO1}, where existence results were obtained via the Galerkin method applied to regularized problems, with solutions constructed as limits of sequences of solutions to these regularized formulations.	In \cite{LFSLSG}, the existence of solutions for a nonstationary Navier-Stokes HVI was demonstrated by developing a temporally semi-discrete approximation scheme, whose approximate solutions converge to a solution of the original problem. The work in \cite{CFWH} addressed the well-posedness of both a nonstationary Stokes HVI and an associated optimal control problem. For the stationary Stokes HVI, well-posedness is further examined in \cite{MLWH1} using an equivalent minimization framework.
	
	More recently, significant efforts have been devoted to the numerical approximation of HVIs in fluid mechanics. In \cite{CFKCWH}, a mixed finite element method is analyzed for a stationary Stokes HVI featuring a nonlinear slip boundary condition, and an optimal order error estimate is derived for the mixed finite element solution. This problem is later revisited in \cite{MLFWWH}, where a nonconforming virtual element method is proposed, and an optimal convergence rate is established for the lowest-order element solution. In \cite{WHKCFJ}, a mixed finite element method is developed for solving a stationary Navier-Stokes HVI with nonlinear slip conditions, with optimal error estimates again shown for the mixed finite element method. The works \cite{FJJLZC,FJWHWY} investigated various discontinuous Galerkin (DG) methods for addressing a variational inequality associated with the stationary Stokes as well as Navier–Stokes equations, subject to a nonlinear slip boundary condition characterized by friction.

	Damping effects are essential in the mathematical modeling of fluid flows with resistance, as they account for key physical phenomena like drag forces, viscous dissipation, and other energy-loss mechanisms (\cite{KWH,CFKCWH,SGMTM,SGKKMTM,MTT}). Such effects are commonly incorporated into the Stokes and Navier-Stokes equations through damping terms. Previous studies have analytically and numerically examined a Stokes variational inequality with damping for viscous incompressible fluids (\cite{HQLM}). Additionally, a minimization principle for the Stokes hemivariational inequality with damping was developed in (\cite{WHHQLM}) (see also \cite{MLWH} for the Stokes hemivariational inequality), where mixed finite element methods were employed to solve the inequality, along with derived error estimates for the numerical solutions. The paper \cite{WWXLCWH} analyzes the stationary Navier-Stokes hemivariational inequality with nonlinear damping, proving well-posedness, developing an optimal mixed finite element method, and confirming the theory with an efficient iterative scheme and numerical tests.	However, the theoretical and numerical analysis of Navier-Stokes hemivariational inequalities for viscous incompressible fluids with damping and pumping effects remain unexplored in the literature. This work bridges that gap by introducing and analyzing Navier-Stokes hemivariational inequalities that incorporate both damping and pumping effects, advancing the understanding of such fluid flow models (Algorithm \ref{alg-num} and Theorem \ref{thm-num-alg}).

		\subsection{Comparisons, novelties and contributions}
In the majority of studies on hemivariational inequalities (HVIs), the existence of solutions is established through abstract surjectivity theorems for pseudomonotone operators. This approach has been successfully applied, for instance, in the analysis of the stationary convective Brinkman-Forchheimer (CBF) HVI (corresponding to (\eqref{eqn-stationary-1} with $\kappa = 0$), where the authors of \cite{JJSGMTM} proved both the existence and uniqueness of weak solutions. However, an alternative and more straightforward approach has been introduced in \cite{WHan,WHan1,WHan2}, providing a different pathway to establish existence results. This method first explores minimization principles for specific classes of HVIs and then generalizes the framework using fixed-point techniques (\cite{WHYY,MLWH} for Navier-Stokes HVI). Notably, this strategy has been successfully applied to both Stokes HVIs and Stokes HVIs with damping in \cite{MLWH1,WHHQLM}, where specialized minimization principles for the Stokes system, with and without damping, were formulated and rigorously analyzed. In this work, we extend this alternative framework by developing a new minimization principle for the Stokes hemivariational inequality that incorporates both damping and pumping effects. Our results broaden the existing theory, enabling a more comprehensive treatment of resistance phenomena in incompressible viscous fluid flows.

This study makes several significant advances in the mathematical modeling of fluid dynamics. We develop a novel Navier-Stokes hemivariational inequality (HVI) framework that simultaneously incorporates both damping and pumping effects in incompressible fluid flows. Our primary theoretical achievement is the establishment of a rigorous minimization principle for the Stokes system combined with damping and pumping effects, which serves two crucial purposes: it provides a foundation for numerical optimization techniques to solve such HVis, while also offering new insights into the mathematical structure of such problems. The minimization principle plays a pivotal role in our analysis, enabling us to prove the well-posedness of the original problem through application of fundamental fixed-point theorems, specifically, either the Banach contraction principle or Schauder's fixed point theorem. Our analytical approach offers distinct advantages for implementation in applied settings. By developing concrete mathematical tools rather than relying on overly abstract theory, we present a framework that is both rigorous and readily applicable for researchers in applied mathematics and engineering disciplines. This practical methodology provides an intuitive pathway to understanding the system's behavior while maintaining mathematical precision,  a crucial balance that enhances both theoretical insight and computational tractability.

While prior work in \cite{WHHQLM} studied Stokes hemivariational inequalities (HVIs) with damping effects, our work first extends this by analyzing the minimization problem for a Stokes HVI that incorporates both damping and pumping effects (Theorem \ref{thm-main-linear}).  The existence and uniqueness results for Navier-Stokes HVIs in \cite{MLWH}  rely on the Banach contraction principle, which imposes a smallness condition on the data \(\f\). In contrast, for the supercritical case  (\(r \in (3, \infty)\)), we establish global existence results without any restrictions on \(\f\) (Theorem \ref{thm-global}).   For \(r > 5\) in three dimensions, Sobolev embedding techniques are insufficient. Instead of the Banach contraction principle, we employ Schauder’s fixed point theorem, enabling broader existence results.  
Due to these methodological constraints, we derive error estimates only for dimensions \(d \in \{2,3\}\) with \(r \in \left[1, \frac{d+2}{d-2}\right]\) (where \(1 \leq r < \infty\) for \(d=2\)) (Theorem \ref{thm-num-alg}). 
We summarize the above discussion in the following table:

	\begin{table}[h]
	\begin{tabular}{|c|c|c|}
		\hline
		Values of $d$ and $r$&Technique&Error estimates\\
		\hline
		$d=2,r\in[1,\infty)$, $d=3$, $r\in[1,5]$& Banach's contraction principle&Yes\\
		\hline
		$d=3$, $r\in(5,\infty)$&Schauder’s fixed point theorem&No\\
		\hline
	\end{tabular}
		\caption{Values of $d$ and $r$, techniques and error estimates}
		\end{table}

The minimization problem addressed in Problem \ref{prob-min}, along with its solution in Theorem \ref{thm-main-linear}, encountered challenges in establishing the strong convexity of the energy functional. For $d = 2$, $r \in [1, \infty)$, and $d = 3$, $r \in [1, 5]$, we were able to directly utilize the strong convexity of the energy functional (Lemma \ref{lem-strong-convex}). However, for the case $d = 3$, $r \in (5, \infty)$, we established a weaker, strong convexity-type property, as discussed in Remark \ref{rem-strong-convex}.

	\subsection{Organization of the paper} The remainder of this work is structured as follows: Section \ref{sec2} presents preliminary material, including key concepts of Clarke's generalized directional derivative and subdifferential calculus, along with their fundamental properties. We then formulate the CBFeD hemivariational inequality (HVI) corresponding to problem \eqref{eqn-stationary-1}-\eqref{eqn-bondary-2} and establish the necessary functional framework.
	
	Section \ref{sec3} addresses the well-posedness analysis of the CBFeD HVI. Our approach combines two key methodologies: first, we employ a minimization principle for auxiliary Stokes HVIs incorporating damping and pumping effects (Problem \ref{prob-min}, Theorem \ref{thm-main-linear}); second, we apply fixed-point theory (Banach or Schauder's fixed point theorems) to establish local well-posedness (Theorem \ref{thm-contra}). The section further extends these results to global well-posedness (Theorem \ref{thm-global}) and demonstrates solution sensitivity through Lipschitz/H\"older continuity with respect to the external forcing data (Theorem \ref{thm-global}).
	
	Notably, the Banach fixed-point framework naturally suggests a constructive iteration scheme for solving the CBFeD HVI, implemented via successive solutions of Stokes HVIs with damping and pumping effects. This approach simultaneously provides both theoretical guarantees and practical computational methodology (Algorithm \ref{alg-num} and Theorem \ref{thm-num-alg}).

	\section{Mathematical Formulation}\label{sec2}\setcounter{equation}{0}

	\subsection{Preliminaries}
	The function spaces appearing in this work are defined over the field of real numbers. Let $\mathbb{X}$  be a normed space. We denote its norm by $\|\cdot\|_{\mathbb{X}}$  its topological dual by $\mathbb{X}^{*}$.
	For simplicity, we assume throughout that $\mathbb{X}$  is a Banach space unless otherwise specified.
	

	We start with the definition of a locally Lipschitz function.
	\begin{definition}
		A function $f:\X\to\mathbb{R}$ is said to be locally Lipschitz if for every $\x\in\X$, there exists a neighborhood $U$ of $\x$ and a constant $L_U$ such that $$|f(\y)-f(\z)|\leq L_U\|\y-\z\|_{\X}\ \text{ for all }\ \y,\z\in U.$$ 
	\end{definition}
	
	Let us now recall the definitions of the generalized directional derivative and the generalized gradient in the sense of Clarke for a locally Lipschitz function.
	\begin{definition}[{\cite[Definition  5.6.3]{ZdSm1}}]
		Let $f :\X \to\R$ be a locally Lipschitz function. The generalized directional derivative of $f$ at $\x\in\X$ in the direction $\v\in\X$, denoted by $f^0(\x;\v)$, is defined by
		\begin{align*}
			f^0(\x;\v)=\lim_{\y\to\x}\sup_{\lambda\downarrow 0}\frac{f(\y+\lambda\v)-f(\y)}{\lambda}. 
		\end{align*}
		The generalized gradient or subdifferential of $f$ at $\x$, denoted by $\partial f(\x)$, is a subset of the dual space $\X^{*}$ given by
		\begin{align*}
			\partial f(\x)=\left\{\zeta\in\X^{*}:f^0(\x;\v)\geq {}_{\X^{*}}\langle\zeta,\v\rangle_{\X}\ \text{ for all }\ \v\in\X \right\}. 
		\end{align*}
		A locally Lipschitz function $f$ is said to be \emph{regular (in the sense of Clarke)} at $\x \in \X$  if for all $v\in\X$, the one-sided directional derivative $f'(\x;\v)$  exists and $f^0(\x;\v) = f'(\x;\v)$.
	\end{definition}
	
	Given the generalized subdifferential, the generalized directional derivative can be computed using the formula provided in {\cite[Proposition 2.1.2]{FHC} (see  \cite[Proposition 5.6.9]{ZdSm1}} also)
	\begin{align}\label{eqn-gen-grad}
		f^0(\x;\v)=\max\left\{\langle\boldsymbol{\zeta},\v\rangle:\boldsymbol{\zeta}\in\partial f(\x)\right\}.
		\end{align}

	
If $f_1,f_2:\X\to \R$  are locally Lipschitz continuous, then according to \cite[Proposition 2.3.3]{FHC}, we have 
\begin{align}\label{eqn-subdiff-1}
	\partial (f_1+f_2)(\bu)\subset \partial f_1(\bu)+\partial f_2(\bu),\ \text{ for all }\ \bu\in\X,
\end{align}
or equivalently,
\begin{align}\label{eqn-subdiff-2}
 (f_1+f_2)^0(\bu;\v)\leq  f_1^0(\bu;\v)+ f_2^0(\bu;\v),  \ \text{ for all }\ \bu,\v\in\X. 
\end{align}

	In the analysis of the CBFeD  hemivariational inequality, the following results will be essential. Given two points $\bu,\v\in\X$, the notation $[\bu,\v]$ signifies the closed line segment consisting of all points $\lambda\bu+(1-\lambda)\v$ for $\lambda\in(0,1)$; $(\bu,\v)$ signifies the open line segment. 
\begin{lemma}[Mean value theorem, {\cite[Theorem 2.3.7]{FHC} }]\label{lem-mvt}
	Let $\bu$ and $\v$ be points in $\X$ and suppose that $f$  is Lipschitz continuous near each point of a non-empty closed convex set containing the line segment $[\bu,\v]$. Then there exists a point $\w\in(\bu,\v)$ such that
	\begin{align*}
		f(\bu)-\f(\v)\in \langle\partial f(\w),\bu-\v\rangle. 
	\end{align*}
\end{lemma}

\begin{lemma}[{\cite[Proposition 3.1, Theorem 3.4]{LFSLSG}}]\label{lem-strong-convex}
	Let $\X$ be a real Banach space and let $h:\X\to\R$ be locally Lipschitz continuous. Then $h$ is \emph{strongly convex} on $\X$ with a constant $\vartheta>0$, that is, for all $\bu,\v\in\X$  and $ \lambda\in[0,1],$
	\begin{align*}
		h(\lambda \bu+(1-\lambda )\v)\leq\lambda h(\bu)+(1-\lambda)h(\v)-\vartheta\lambda(1-\lambda)\|\bu-\v\|_{\X}^2,
	\end{align*}
	if and only if $\partial h$ is strongly monotone with a constant $2\vartheta$, that is, 
	\begin{align*}
		\langle\boldsymbol{\xi}-\boldsymbol{\eta}, \bu-\v\rangle\geq 2\vartheta\|\bu-\v\|_{\X}^2, \ \text{ for all }\ \bu,\v\in\X, \ \boldsymbol{\xi}\in\partial h(\bu),\  \boldsymbol{\eta} \in\partial h(\v). 
	\end{align*}
\end{lemma}

Depending on different values of $r$, we use Banach contraction mapping principle or Schauder’s fixed point theorem to obtain the required results. 
Let $\X$ be a Banach space.
\begin{theorem}[{\cite[Theorem 1, page 502]{LCE}}, Banach fixed point theorem]\label{thm-BFT}
	Assume $\mathcal{G}:\X\to\X$ is a nonlinear mapping and suppose that $\mathcal{G}$ is a contraction, that is, 
	\begin{align*}
		\|\mathcal{G}(\bu)-\mathcal{G}(\widetilde{\bu})\|_{\X}\leq\rho\|\bu-\widetilde{\bu}\|_{\X}\ \text{ for any }\ \bu,\widetilde{\bu}\in\X,
	\end{align*}
	for some $0<\rho<1$. Then $\mathcal{G}$ has a unique fixed point.
\end{theorem}

\begin{theorem}[{\cite[Theorem 4, page 507]{LCE}} Schauder’s fixed point theorem]\label{thm-SFT}
	Let $\mathcal{K}\subset\X$  be closed, bounded, convex, and nonempty. Then any compact operator $\mathcal{G}:\mathcal{K}\to\mathcal{K}$  has at least one fixed point.	
\end{theorem}

	\subsection{Functional setting} 
		We begin by introducing the following function space:
		\begin{align}
			\mathscr{M} = \left\{\boldsymbol{v} \in \C^{\infty}(\overline{\mathcal{O}};\mathbb{R}^d) : \nabla \cdot \boldsymbol{v} = 0 \text{ in } \mathcal{O}, \boldsymbol{v}_{\tau} = \boldsymbol{0} \text{ on } \Gamma\right\}.
		\end{align}
		The tangential condition $\boldsymbol{v}_{\tau} = \boldsymbol{0}$ on $\Gamma$ implies $\boldsymbol{v} = v_n \boldsymbol{n}$, where $v_n = \boldsymbol{v} \cdot \boldsymbol{n}$. Consequently, we obtain  
		\begin{align}
			\boldsymbol{v} \times \boldsymbol{n} = \boldsymbol{0} \text{ on } \Gamma.
		\end{align}
		Let $\mathbb{V}$ and $\mathbb{H}$ denote the closures of $\mathscr{M}$ with respect to the norms of $\mathrm{H}^1(\mathcal{O};\mathbb{R}^d)$ and $\mathrm{L}^2(\mathcal{O};\mathbb{R}^d)$, respectively. These spaces admit the following characterizations:
	\begin{align}
	\mathbb{H} &= \left\{\boldsymbol{v} \in \mathrm{L}^2(\mathcal{O};\mathbb{R}^d) : \nabla \cdot \boldsymbol{v} = 0 \text{ in } \mathcal{O}, \boldsymbol{v} \times \boldsymbol{n} = \boldsymbol{0} \text{ on } \Gamma\right\}, \\
\mathbb{V} &= \left\{\boldsymbol{v} \in \mathrm{H}^1(\mathcal{O};\mathbb{R}^d) : \nabla \cdot \boldsymbol{v} = 0 \text{ in } \mathcal{O}, \boldsymbol{v} \times \boldsymbol{n} = \boldsymbol{0} \text{ on } \Gamma\right\}.
\end{align}
For $p \in (2, \infty)$, we define $\widetilde{\mathbb{L}}^p$ as the closure of $\mathscr{M}$ with respect to the $\mathrm{L}^p(\mathcal{O};\mathbb{R}^d)$ norm. This space can be characterized analogously to the above. We then have the following continuous and dense embeddings:
\begin{align}
\mathbb{V} \cap \wi{\mathbb{L}}^p \hookrightarrow \mathbb{V} \hookrightarrow \mathbb{H} \equiv \mathbb{H}^{*} \hookrightarrow \mathbb{V}^{\prime} \hookrightarrow \mathbb{V}^{*} + \wi{\mathbb{L}}^{\frac{p}{p-1}},
\end{align}
where $\V^*$ denotes the dual space of $\V$. Here, the embedding $\iota: \mathbb{V} \to \mathbb{H}$ is both continuous and compact.  The norms on $\mathbb{H}$ and $\widetilde{\mathbb{L}}^p$ (for $p \in (2, \infty)$) are given by  
\begin{align*}
\|\boldsymbol{u}\|_{\mathbb{H}} = \left(\int_{\mathcal{O}} |\boldsymbol{u}(\boldsymbol{x})|^2 \mathrm{d}\boldsymbol{x}\right)^{1/2}, \quad \|\boldsymbol{u}\|_{\mathbb{L}^p} = \left(\int_{\mathcal{O}} |\boldsymbol{u}(\boldsymbol{x})|^p \mathrm{d}\boldsymbol{x}\right)^{1/p},
\end{align*}
while the norm on $\mathbb{V}$ is the $\mathbb{H}^1$-norm:  
\begin{align*}
\|\boldsymbol{u}\|_{\mathbb{H}^1} = \left(\|\boldsymbol{u}\|_{\mathbb{H}}^2 + \|\nabla \boldsymbol{u}\|_{\mathbb{H}}^2\right)^{1/2}.
\end{align*}
We now establish that the norm  
\begin{align*}
\|\boldsymbol{u}\|_{\mathbb{V}} = \left(\int_{\mathcal{O}} |\mathrm{curl} \, \boldsymbol{u}(\boldsymbol{x})|^2 \mathrm{d}\boldsymbol{x}\right)^{1/2} = \|\mathrm{curl} \, \boldsymbol{u}\|_{\mathbb{H}}
\end{align*}
is equivalent to the $\mathbb{H}^1$-norm on $\mathbb{V}$.  

In what follows, the duality pairings between $\mathbb{V}$ and $\mathbb{V}^{*}$, $\wi{\mathbb{L}}^p$ and $\wi{\mathbb{L}}^{\frac{p}{p-1}}$, and $\mathbb{V} \cap \wi{\mathbb{L}}^p$ and $\mathbb{V}^{*} + \wi{\mathbb{L}}^{\frac{p}{p-1}}$ will all be denoted by $\langle \cdot, \cdot \rangle$.

The following Green's formulae  from \cite[Equations (2.17) and (2.22)]{VGPR}, will be useful: For all $\boldsymbol{v} \in \mathbb{H}^1(\mathcal{O})$ and $\boldsymbol{\phi} \in \mathbb{H}^1(\mathcal{O})$,  
\begin{align}\label{eqn-green-2}
	(\mathrm{curl} \, \boldsymbol{v}, \boldsymbol{\phi}) - (\boldsymbol{v}, \mathrm{curl} \, \boldsymbol{\phi}) = {}_{\mathbb{H}^{-\frac{1}{2}}(\Gamma)}\langle \boldsymbol{v} \times \boldsymbol{n}, \boldsymbol{\phi} \rangle_{\mathbb{H}^{\frac{1}{2}}(\Gamma)},
\end{align}
where $\mathbb{H}^{-\frac{1}{2}}(\Gamma)$ denotes the dual of $\mathbb{H}^{\frac{1}{2}}(\Gamma)$.

	\begin{lemma}[{\cite[Lemma 2.12]{JJSGMTM}}]\label{lem-equiv}
		There exist positive constants $C_1$ and $C_2$ such that 
		\begin{align}\label{eqn-equiv}
			\|\bu\|_{\H^1}\leq C_1\left(\|\bu\|_{\H}+\|\mathrm{curl \ }\bu\|_{\H}\right), \ \text{ for all }\ \bu\in\V,
		\end{align}
		and 
		\begin{align}\label{eqn-equiv-1}
			\|\bu\|_{\H}\leq C_2\|\mathrm{curl \ }\bu\|_{\H}, \ \text{ for all }\ \bu\in\V.
		\end{align}
	\end{lemma}

	\begin{remark}\label{rem-equiv}
		Since $\|\bu\|_{\H^1} = \big(\|\bu\|_{\H}^2 + \|\nabla\bu\|_{\H}^2\big)^{1/2}$, there exists a constant $C_3 > 0$ such that  
		\begin{align*}  
			C_3 \|\mathrm{curl\,}\bu\|_{\H} \leq \|\bu\|_{\H^1}, \  \text{ for all }\ \bu \in \V.  
		\end{align*}  
		Furthermore, by applying the relations \eqref{eqn-equiv} and \eqref{eqn-equiv-1}, we deduce the existence of another constant $C_4 > 0$ satisfying  
		\begin{align*}  
			\|\bu\|_{\H^1} \leq C_4 \|\mathrm{curl\,}\bu\|_{\H}, \quad \forall \bu \in \V.  
		\end{align*}  
		Consequently, the norms $\|\bu\|_{\H^1}$ and $\|\mathrm{curl\,}\bu\|_{\H}$ are equivalent on the space $\V$.  
	\end{remark}

By applying the Sobolev trace theorem, we obtain the following estimate:  
\begin{align}\label{eqn-eigen}  
	\|v_n\|_{\L^2(\Gamma)} \leq \lambda_0^{-1/2} \|\v\|_{\V}, \ \text{ for all }\  \v \in \V,  
\end{align}  
where $\lambda_0 > 0$ is the principal eigenvalue of the problem:  
\begin{align}\label{eqn-eigen-1}  
	\bu \in \V, \  \int_{\mathcal{O}} \mathrm{curl \ }\bu(\x) \cdot \mathrm{curl \ }\v(\x) \, \d\x = \lambda \int_{\Gamma} u_n(\x) v_n(\x)\d\S(\x), \ \text{ for all }\  \v \in \V.  
\end{align}

	\subsection{Linear operator}
	We define the bilinear form \( a : \V \times \V \to \R \) as  
	\[ a(\bu, \v) := (\mathrm{curl\ } \bu, \mathrm{curl\ } \v), \  \text{for } \ \bu, \v \in \V. \]  
	From this definition, it follows that \( a(\cdot, \cdot) \) is \(\V\)-continuous, meaning  
	\begin{align}\label{eqn-abound}  
		|a(\bu, \v)| \leq C \|\bu\|_{\V} \|\v\|_{\V}, \  \text{ for all }\  \bu, \v \in \V,  
	\end{align}  
	for some constant \( C > 0 \). Consequently, by the \emph{Riesz representation theorem}, there exists a unique linear operator \( \mathcal{A} : \V \to \V^{*} \) such that  
	\begin{align}\label{eqn-def-A}  
		a(\bu, \v) = \langle \mathcal{A} \bu, \v \rangle, \  \text{ for all } \ \bu, \v \in \V.  
	\end{align}  
		Furthermore, the bilinear form \( a(\cdot, \cdot) \) is \(\V\)-coercive, satisfying  
	\begin{align}\label{eqn-coercive}  
		a(\bu, \bu) = \|\bu\|_{\V}^2 = \|\mathrm{curl\ } \bu\|_{\H}^2, \  \text{for all } \ \bu \in \V.  
	\end{align}  
	Thus, by the \emph{Lax-Milgram theorem} (due to \eqref{eqn-abound} and \eqref{eqn-coercive}), the operator \( \mathcal{A} : \V \to \V^{*} \) is an isomorphism.

	\subsection{Bilinear operator}
	Let us define the \emph{trilinear form} $b(\cdot,\cdot,\cdot):\V\times\V\times\V\to\R$ by $$b(\bu,\v,\w)=\int_{\mathcal{O}}(\mathrm{curl\ }\bu(\x)\times\v(\x))\cdot\w(\x)\d \x.$$
	It can be easily seen that for all $\bu,\v,\w\in\V$
	\begin{align}\label{eqn-bbound}
		|b(\bu,\v,\w)|\leq\|\mathrm{curl\ }\bu\|_{\H}\|\v\|_{\L^4}\|\w\|_{\L^4}\leq C_b\|\bu\|_{\V}\|\v\|_{\H^1}\|\w\|_{\H^1}\leq C_b\|\bu\|_{\V}\|\v\|_{\V}\|\w\|_{\V},
	\end{align}
	where we have used Sobolev's embedding also. Therefore, the linear map $b(\bu, \v, \cdot) $ is continuous on $\V$, the corresponding element of $\V^{*}$ is denoted by $\mathcal{B}(\bu, \v),$ so that \begin{align}\label{eqn-def-B} 
		 b(\v,\v,\w)=\langle\mathcal{B}(\bu,\v),\w\rangle, \ \text{ for all }\ \bu,\v,\w\in\V.
		 \end{align}
	Note that $\mathcal{B}(\cdot,\cdot)$ is a bilinear operator. We also denote  $\mathcal{B}(\bu) = \mathcal{B}(\bu, \bu)$. 
		Applying the formula \eqref{eqn-green-2} by replacing $\boldsymbol{\phi}$ with $\v\times\w$ and $\v$ with $\bu$, we find  for all $\bu,\v,\w\in\V$ that 
	\begin{align}\label{eqn-b-int}
		&b(\bu,\v,\w)\nonumber\\&=\int_{\mathcal{O}}(\mathrm{curl \ }\bu(\x)\times\v(\x))\cdot\w(\x)\d\x 
		\nonumber\\&=\int_{\mathcal{O}}\mathrm{curl \ }\bu(\x)\cdot(\v(\x)\times\w(\x))\d\x \nonumber\\&= \int_{\mathcal{O}}[\bu(\x)\cdot\mathrm{curl\ }(\v(\x)\times\w(\x))]\d\x+{}_{\H^{-\frac{1}{2}}(\Gamma)}\langle\bu\times\n, \v\times\w\rangle_{\H^{\frac{1}{2}}(\Gamma)}\nonumber\\&= \int_{\mathcal{O}}\bu(\x)\cdot[\v(\x)(\nabla\cdot\w(\x))-\w(\x)(\nabla\cdot\v(\x))+(\w(\x)\cdot\nabla)\v(x)-(\v(\x)\cdot\nabla)\w(\x)]\d\x\nonumber\\&= \int_{\mathcal{O}}\bu(\x)\cdot[(\w(\x)\cdot\nabla)\v(\x)-(\v(\x)\cdot\nabla)\w(\x)]\d\x
		\nonumber\\&=(\bu,(\w\cdot\nabla)\v-(\v\cdot\nabla)\w),
	\end{align}
	where we have used the fact that $\bu\times\n=\boldsymbol{0}$ on $\Gamma$ and $\nabla\cdot\w=\nabla\cdot\v=0$. 	By taking $\v=\w$, we find
	\begin{align}\label{eqn-b-est}
		b(\bu,\v,\w)=-b(\bu,\w,\v), \ \text{ for all }\ \bu,\v,\w\in\V. 
	\end{align}
	Taking $\v=\w$ in \eqref{eqn-b-est}, we immediately have 
	\begin{align}\label{eqn-b0}
		b(\bu,\v,\v)=\langle\mathcal{B}(\bu,\v),\v\rangle=0,\ \text{ for all }\ \bu,\v\in\V. 
	\end{align}
	For $r>3$ and $\bu,\v\in\V\cap\widetilde{\L}^{r+1}$, by using H\"older's and Young's inequalities, we estimate $|\langle\mathcal{B}(\bu,\bu),\v\rangle|$ as 
	\begin{align}\label{eqn-b-est-lr}
		|\langle\mathcal{B}(\bu,\bu),\v\rangle|&\leq\|\bu\|_{\V}\|\bu\|_{\L^{\frac{2(r+1)}{r-1}}}\|\v\|_{\L^{r+1}}\leq\|\bu\|_{\V}\|\bu\|_{\H}^{\frac{r-3}{r-1}}\|\bu\|_{\L^{r+1}}^{\frac{2}{r-1}}\|\v\|_{\L^{r+1}}, 
	\end{align}
	so that the operator $\B:\V\cap\wi{\L}^{r+1}\to\V^{*}+\wi{\L}^{\frac{r+1}{r}}$.

	\subsection{Nonlinear operator}
	Let us now consider the operator  $$\mathcal{C}_r(\bu):=|\bu|^{r-1}\bu, \ \text{ for }\ \bu\in\wi{\L}^{r+1}. $$  For convenience of notation, we use $\mathcal{C}$ for $\mathcal{C}_r$ in the rest of the paper.  It is immediate that \begin{align}\label{eqn-c}\langle\mathcal{C}(\bu),\bu\rangle =\|\bu\|_{\L^{r+1}}^{r+1}\end{align} and the map $\mathcal{C}(\cdot):\wi{\L}^{r+1}\to\wi{\L}^{\frac{r+1}{r}}$ is Gateaux differentiable with its Gateaux derivative 
	\begin{align}\label{Gaetu}
		\mathcal{C}'(\bu)\v&=\left\{\begin{array}{cl}\v,&\text{ for }r=1,\\ \left\{\begin{array}{cc}|\bu|^{r-1}\v+(r-1)\left(\frac{\bu}{|\bu|^{3-r}}(\bu\cdot\v)\right),&\text{ if }\bu\neq \mathbf{0},\\\mathbf{0},&\text{ if }\bu=\mathbf{0},\end{array}\right.&\text{ for } 1<r<3,\\ |\bu|^{r-1}\v+(r-1)\bu|\bu|^{r-3}(\bu\cdot\v), &\text{ for }r\geq 3,\end{array}\right.
	\end{align}
	for all $\v\in\wi{\L}^{r+1}$. 

	\begin{lemma}[{\cite[Section 2.4]{SGMTM}}]
		For all 	$\bu,\v\in\wi{\L}^{r+1}(\mathcal{O})$ and $r\geq 1$, we have 
		\begin{align}\label{2.23}
			&\langle\bu|\bu|^{r-1}-\v|\v|^{r-1},\bu-\v\rangle\geq \frac{1}{2}\||\bu|^{\frac{r-1}{2}}(\bu-\v)\|_{\H}^2+\frac{1}{2}\||\v|^{\frac{r-1}{2}}(\bu-\v)\|_{\H}^2\geq 0,
		\end{align}
		and 
		\begin{align}\label{Eqn-mon-lip}
			&\langle\bu|\bu|^{r-1}-\v|\v|^{r-1},\bu-\v\rangle
			\geq \frac{1}{2^{r-1}}\|\bu-\v\|_{\L^{r+1}}^{r+1}.
		\end{align}
	\end{lemma}
	
	Similar results hold true for the operator $\mathcal{C}_0(\bu):=|\bu|^{q-1}\bu$, for $\bu\in\wi{\L}^{r+1},$ $q\in [1,r)$.  In the sequel, we use the notation
	\begin{align}\label{eqn-def-C}  
		c(\bu,\v)=\langle\mathcal{C}(\bu),\v\rangle\ \text{ and }\ 	c_0(\bu,\v)=\langle\mathcal{C}_0(\bu),\v\rangle, \ \text{ for all }\ \bu,\v\in\wi{\L}^{r+1}. 
	\end{align}

	\subsection{Abstract formulation } Multiplying the equation of motion \eqref{eqn-stationary-1} by $\v \in\V\cap\wi{\L}^{r+1}$ and applying  Green's formula (see \eqref{eqn-green-2}), we obtain 
	\begin{equation}\label{eqn-absract-1}
	\langle\mu\mathcal{A}\bu+\mathcal{B}(\bu)+\alpha\bu+\beta\mathcal{C}(\bu)+\kappa\mathcal{C}_0(\bu),\v\rangle+\int_{\Gamma}qv_n\d\Gamma=\langle\f,\v\rangle,
	\end{equation}
	for all $\v\in\V\cap\wi{\L}^{r+1}$. From the relation \eqref{eqn-bondary-2}, by using the definition of the Clarke subdifferential, we have
	\begin{align}\label{eqn-clarke-sub}
		\int_{\mathcal{O}}qv_n\d\Gamma\leq \int_{\Gamma}j^0(u_n;v_n)\d\Gamma,
	\end{align}
	where $j^0(\xi;\zeta)\equiv j^0(\x,\xi;\zeta)$ denotes the generalized directional derivative of $j(\x,\cdot)$ at the point $\xi\in\R$ in the direction $\zeta\in\R$. 
	
	 The relations \eqref{eqn-absract-1} and \eqref{eqn-clarke-sub} yield the following variational formulation: 
	\begin{problem}\label{prob-hemi-inequality}
		Find $\bu\in\V\cap\wi{\L}^{r+1}$ such that  
		\begin{equation}\label{eqn-absract-2}
			\langle\mu\mathcal{A}\bu+\mathcal{B}(\bu)+\alpha\bu+\beta\mathcal{C}(\bu)+\kappa\mathcal{C}_0(\bu),\v\rangle+\int_{\Gamma}j^0(u_n;v_n)\d\Gamma\geq \langle\f,\v\rangle,
		\end{equation}
		for all $\v\in\V\cap\wi{\L}^{r+1}$. 
	\end{problem}
	We make the following assumptions on the superpotential $j$:
	\begin{hypothesis}\label{hyp-new-j}
		The superpotential $j:\Gamma\times\mathbb{R}\to\mathbb{R}$ satisfy 
		\begin{enumerate}
			\item [(H.1)] $j(\cdot,\xi)$ is measurable on $\Gamma$ for all $\xi\in\mathbb{R}$ and $j(\cdot,0)\in  \mathbb{L}^1(\Gamma)$;
			\item [(H.2)] $j(\x,\cdot)$ is locally Lipschitz on $\mathbb{R}$ for a.e. $\x \in \Gamma$;
			\item [(H.3)] $|\zeta|\leq c_0+c_1|\xi|$ for all $\zeta\in\partial j(\x,\xi)$, $\xi\in\mathbb{R}$ for a.e. $\x\in\Gamma$ with $c_0,c_1\geq 0$; 
			\item [(H.4)] $(\zeta_1-\zeta_2)\cdot(\xi_1-\xi_2)\geq -m|\xi_1-\xi_2|^2$ for all $\zeta_i\in \partial j(\x,\xi_i)$, $\xi_i\in\mathbb{R}$, $i=1,2,$ for a.e. $\x\in\Gamma$ with $m\geq 0$. 
		\end{enumerate}
	\end{hypothesis}
Using \eqref{eqn-gen-grad} and Hypothesis (H.3), we find 
\begin{align}
	|j^0(\xi_1;\xi_2)|\leq \left(c_0+c_1|\xi_1|\right)|\xi_2|, \ \text{ for all }\ \xi_1,\xi_2\in\R. 
\end{align}
Condition (H.4) is commonly referred to in the literature as a relaxed monotonicity condition (see \cite[Definition 3.49]{SMAOMS}). It can also be equivalently formulated as follows:
\begin{align}\label{eqn-alternative-con}
	j^0(\xi_1;\xi_2-\xi_1)+	j^0(\xi_2;\xi_1-\xi_2)\leq m|\xi_1-\xi_2|^2,\ \text{ for all }\ \xi_1,\xi_2\in\R.  
\end{align}
Note that this condition holds with $m=0$, if $j:\V\to\R$ is a convex function. 
	Let us define the functional $J:\V\to\R$ by 
	\begin{align}\label{eqn-def-J}
		J(\v):=\int_{\Gamma}j(\x,v_n(\x))\d\S(\x).
	\end{align}
	Then by \cite[Lemma 13]{SMAO}, under Hypothesis \ref{hyp-new-j}, the functional 
	$J$  is well-defined and \emph{locally Lipschitz continuous} on $\V$, and
	\begin{align}\label{eqn-sub-1}
		J^0(\bu;\v)\leq \int_{\Gamma}j^0(u_n;v_n)\d\S(\x) \ \text{ for all }\ \v\in\V. 
	\end{align}
	Combining Hypothesis \ref{hyp-new-j} (H.4), \eqref{eqn-eigen}, and \eqref{eqn-sub-1}, we deduce for all $\v_1,\v_2\in\V$ that 
	\begin{align}\label{eqn-j0-diff}
		J^0(\v_1;\v_2-\v_1)+J^0(\v_2;\v_1-\v_2)&\leq\int_{\Gamma}\left[j^0(v_{1,n};v_{2,n}-v_{1,n})+j^0(v_{2,n};v_{1,n}-v_{2,n})\right]\d S(\x)\nonumber\\&\leq \int_{\Gamma}m|v_{1,n}(\x)-v_{2,n}(\x)|^2\d S(\x)\nonumber\\&\leq m\lambda_0^{-1}\|\v_1-\v_2\|_{\V}^2. 
	\end{align}
	Moreover, it is known from \cite[Theorem 4.20]{SMAOMS} that the condition \eqref{eqn-j0-diff} is equivalent to 
	\begin{align}\label{eqn-j0-diff-1}
		\langle\boldsymbol{\eta}_1-\boldsymbol{\eta}_2,\v_1-\v_2\rangle\geq -m\lambda_0^{-1}\|\v_1-\v_2\|_{\V}^2,\ \text{ for all }\ \v_i\in\V,\ \boldsymbol{\eta}_i\in\partial J(\v_i),\ i=1,2. 
	\end{align}

	Using \eqref{eqn-def-A}, \eqref{eqn-def-B}  and \eqref{eqn-def-C} in \eqref{prob-hemi-inequality}, we deduce CBFeD HVI as

	\begin{problem}\label{prob-hemi-var}
		Find $\bu\in\V\cap\wi{\L}^{r+1}$ such that  
		\begin{align}\label{eqn-concerete-1}
			\mu a(\bu,\v)+\alpha a_0(\bu,\v)&+b(\bu,\bu,\v)+\beta c(\bu,\v)+\kappa c_0(\bu,\v)\nonumber\\+\int_{\Gamma}j^0(u_n;v_n)\d\Gamma&\geq \langle\f,\v\rangle, \ \text{	for all }\  \v\in\V\cap\wi{\L}^{r+1}. 
		\end{align}
	\end{problem}
	
		\section{CBFeD Hemivariational Inequality}\label{sec3}\setcounter{equation}{0}
	In this section, we focus on establishing the well-posedness of the CBFeD HVI (Problem \ref{eqn-concerete-1}) by integrating two primary analytical techniques. First, we develop a minimization framework tailored to auxiliary Stokes-type HVIs that account for both damping and pumping phenomena. Second, we utilize fixed-point arguments, relying on either Banach’s or Schauder’s fixed point theorems, to demonstrate local well-posedness of the original problem. Building on this foundation, the section advances to prove global well-posedness and further investigates the continuous dependence of solutions on external forcing data, establishing Lipschitz or Hölder continuity results.
	
\subsection{An auxiliary problem}
We begin by introducing and analyzing an auxiliary problem. Let $\w\in\V$ be given.
	\begin{problem}\label{prob-hemi-1}
		Find $\bu\in\V\cap\wi{\L}^{r+1}$ such that 
		\begin{align}
			&a(\bu,\v)+\alpha a_0(\bu,\v)+\beta c(\bu,\v)+\kappa c_0(\bu,\v)+J^0(\bu;\v)\nonumber\\&\geq \langle\f,\bu\rangle-b(\w,\w,\v)-\kappa c_0(\w,\v), \ \text{ for all }\ \v\in\V\cap\wi{\L}^{r+1}. 
		\end{align}
	\end{problem}
	We adopt the approach presented in \cite{WHan,WHan1,WHan2,WHHQLM,MLWH} and investigate the existence of a solution to Problem \ref{prob-hemi-1} via an equivalent minimization problem. To this end, we the functional $\boldsymbol{\Psi}:\V\cap\wi{\L}^{r+1}\to\R$ by 
	\begin{align*}
		\langle\boldsymbol{\Psi},\v\rangle=\langle\f,\bu\rangle-b(\w,\w,\v), \ \text{ for all }\ \v\in\V\cap\wi{\L}^{r+1}.  
	\end{align*}
Let us now 	introduce the following energy functional  for all $\v\in\V\cap\wi{\L}^{r+1},$
\begin{align}
	\mathcal{E}(\v)&=\frac{\mu}{2}a(\v,\v)+\frac{\alpha}{2}a_0(\v,\v)+\frac{\beta}{r+1}c(\v,\v)+\frac{\kappa}{q+1}c_0(\v,\v)+J(\v)- 	\langle\boldsymbol{\Psi},\v\rangle\nonumber\\&=\frac{\mu}{2}\|\v\|_{\V}^2+\frac{\alpha}{2}\|\v\|_{\H}^2+\frac{\beta}{r+1}\|\v\|_{\L^{r+1}}^{r+1}+\frac{\kappa}{q+1}\|\v\|_{\L^{q+1}}^{q+1}+J(\v)- 	\langle\boldsymbol{\Psi},\v\rangle,
\end{align}
and consider a corresponding minimization problem:
\begin{problem}\label{prob-min}
	Find $\bu\in\V\cap\wi{\L}^{r+1}$ such that 
	\begin{align}
		\mathcal{E}(\bu)=\inf\left\{\mathcal{E}(\v):\v\in\V\cap\wi{\L}^{r+1}\right\}. 
	\end{align}
\end{problem}
The following result establishes the equivalence between Problem \ref{prob-hemi-1} and Problem \ref{prob-min}. 
\begin{theorem}\label{thm-main-linear}
Let  Hypothesis \ref{hyp-new-j} be satisfied and  suppose  $\f\in \V^*$,
\begin{align}\label{eqn-small}
	m<\mu\lambda_0 \ \text{ and }\ \alpha>\varrho_r,
\end{align}
where \begin{align}\label{eqn-rho-2}
	\varrho_{r}=2\left(\frac{r-q}{r-1}\right)\left(\frac{4(q-1)}{\beta(r-1)}\right)^{\frac{q-1}{r-q}}\left( |\kappa| q2^{q-1}\right)^{\frac{r-1}{r-q}}. 
\end{align}
Then for any given $\w\in\V\cap\wi{\L}^{r+1}$,  Problems \ref{prob-hemi-1} and \ref{prob-min} admit the same unique solution $\bu\in\V\cap\wi{\L}^{r+1}$. 
\end{theorem}
\begin{proof}
	We divide the proof into the following steps:
	\vskip 0.2cm
	\noindent\textbf{Step 1:} \emph{$\mathcal{E}:\V\cap\wi{\L}^{r+1}\to\R$ is locally Lipschitz continuous.}
We first show that $\mathcal{E}:\V\cap\wi{\L}^{r+1}\to\R$ is locally Lipschitz continuous. Let $\bu,\v\in\V\cap\wi{\L}^{r+1}$ be such that $\|\bu\|_{\V\cap\wi{\L}^{r+1}}, \|\v\|_{\V\cap\wi{\L}^{r+1}} \leq R$ for some $R>0$. Then, we use Taylor's formula to find 
\begin{align}
	|\mathcal{E}(\bu)-\mathcal{E}(\v)|&=\Big|\frac{\mu}{2}\left(\|\bu\|_{\V}^2-\|\v\|_{\V}^2\right)+\frac{\alpha}{2}\left(\|\bu\|_{\H}^2-\|\v\|_{\H}^2\right)+\frac{\beta}{r+1}\left(\|\bu\|_{\L^{r+1}}^{r+1}-\|\v\|_{\L^{r+1}}^{r+1}\right)\nonumber\\&\quad+\frac{\kappa}{q+1}\left(\|\bu\|_{\L^{q+1}}^{q+1}-\|\v\|_{\L^{q+1}}^{q+1}\right)+[J(\bu)-J(\v)]- 	\langle\boldsymbol{\Psi},\bu-\v\rangle\Big|\nonumber\\&\leq\frac{\mu}{2}\left(\|\bu\|_{\V}+\|\v\|_{\V}\right)\|\bu-\v\|_{\V}+\frac{\alpha}{2}\left(\|\bu\|_{\H}+\|\v\|_{\H}\right)\|\bu-\v\|_{\H}\nonumber\\&\quad+\beta \int_0^1\|\theta|\bu|+(1-\theta)|\v|\|_{\L^{r+1}}^{r}\|\bu-\v\|_{\L^{r+1}}\d\theta \nonumber\\&\quad+|\kappa| \int_0^1\|\theta|\bu|+(1-\theta)|\v|\|_{\L^{q+1}}^{q}\|\bu-\v\|_{\L^{q+1}}\d\theta+C_R\|\bu-\v\|_{\L^2(\Gamma)}\nonumber\\&\quad+\|\Psi\|_{\V^*}\|\bu-\v\|_{\V}\nonumber\\&\leq\left[\max\left\{(\mu R+C\alpha R+C_R\lambda_0^{-1/2}+\|\Psi\|_{\V^*}), (\beta 2^rR^r+ |\kappa|2^qR^q|\mathcal{O}|^{\frac{r-q}{r+1}}) \right\}\right]\nonumber\\&\quad\times\|\bu-\v\|_{\V\cap\wi{\L}^{r+1}},
\end{align}
so that the  locally Lipschitz property follows. 
	\vskip 0.2cm
\noindent\textbf{Step 2:} \emph{$\mathcal{E}:\V\cap\wi{\L}^{r+1}\to\R$ is coercive.}
 We infer from \eqref{eqn-subdiff-1} that 
\begin{align}
	\partial\mathcal{E}(\bv)\subset \A\bv+\alpha\bv+\beta\mathcal{C}(\bv)+\kappa\mathcal{C}_0(\bv)+\partial J(\v)-\Psi. 
\end{align}
Let us now show the coercive property of the operator $\partial\mathcal{E}:\V\cap\wi{\L}^{r+1}\to 2^{\V^{*}+\L^{\frac{r+1}{r}}}.$ For any $\bv\in\V\cap\wi{\L}^{r+1}$ and $\boldsymbol{\xi}\in\partial\mathcal{E}(\bv)$, we denote  
\begin{align}
	\boldsymbol{\xi}= \A\bv+\alpha\bv+\beta\mathcal{C}(\bv)+\kappa\mathcal{C}_0(\bv)+\boldsymbol{\eta}-\Psi, \ \boldsymbol{\eta}\in\partial J(\bv). 
\end{align}
Then, it can be easily seen that 
\begin{align}\label{eqn-min-2}
	\langle	\boldsymbol{\xi},\bv\rangle=\mu\|\bv\|_{\V}^2+\alpha\|\bv\|_{\H}^2+\beta\|\bv\|_{\L^{r+1}}^{r+1}+\kappa\|\bv\|_{\L^{q+1}}^{q+1}+\langle\boldsymbol{\eta},\bv\rangle-\langle\Psi,\v\rangle. 
\end{align}
Using H\"older's and Young's inequality, we find 
\begin{align}\label{eqn-pump-est}
	\kappa	\|\bv\|_{\L^{q+1}}^{q+1}&=|\kappa|\int_{\mathcal{O}}|\bv(\x)|^{q+1}\d\x\leq|\kappa||\mathcal{O}|^{\frac{r-q}{r+1}}\bigg(\int_{\mathcal{O}}|\bv(\x)|^{r+1}\d\x\bigg)^{\frac{q+1}{r+1}}\nonumber\\&=|\kappa||\mathcal{O}|^{\frac{r-q}{r+1}}\|\bv\|_{\L^{r+1}}^{q+1}\leq\frac{\beta}{2}\|\bv\|_{\L^{r+1}}^{r+1}+|\kappa|^{\frac{r+1}{r-q}}|\mathcal{O}|, 
\end{align}
where $ |\mathcal{O}|$ denotes the Lebesgue measure of $\mathcal{O}$. 
We estimate the final term in the right hand side of \eqref{eqn-min-2} by using \eqref{eqn-j0-diff-1} and Hypothesis \ref{hyp-new-j} (H3) as 
\begin{align}\label{eqn-min-3}
	\langle\boldsymbol{\eta},\bv\rangle&=\langle\boldsymbol{\eta}-\boldsymbol{\zeta},\bv-\boldsymbol{0}\rangle+\langle\boldsymbol{\zeta},\bv\rangle\nonumber\\&\geq -m\lambda_0^{-1}\|\v\|_{\V}^2-\|\boldsymbol{\zeta}\|_{\L^{2}(\Gamma)}\|\bv\|_{\L^{2}(\Gamma)}\nonumber\\&\geq -m\lambda_0^{-1}\|\v\|_{\V}^2-c_0|\Gamma|^{1/2}\lambda_0^{-1/2}\|\v\|_{\V},
\end{align}
where $\boldsymbol{\zeta}\in\partial J(\boldsymbol{0})$. Using the Cauchy-Schwarz inequality, we find
\begin{align}\label{eqn-min-4}
	|\langle\Psi,\v\rangle|\leq\|\Psi\|_{\V^*+\L^{\frac{r+1}{r}}}\|\v\|_{\V\cap\L^{r+1}}. 
\end{align}
Using \eqref{eqn-pump-est}, \eqref{eqn-min-3} and \eqref{eqn-min-4} in \eqref{eqn-min-2}, we deduce
\begin{align}\label{eqn-min-5}
	\langle\boldsymbol{\xi},\bv\rangle&\geq\left(\mu-m\lambda_0^{-1}\right)\|\bv\|_{\V}^2+\alpha\|\bv\|_{\H}^2+\frac{\beta}{2}\|\bv\|_{\L^{r+1}}^{r+1}-|\kappa|^{\frac{r+1}{r-q}}|\mathcal{O}|-c_0|\Gamma|^{1/2}\lambda_0^{-1/2}\|\bv\|_{\V}\nonumber\\&\quad-\|\Psi\|_{\V^*+\L^{\frac{r+1}{r}}}\|\v\|_{\V\cap\L^{r+1}}.  
\end{align}
From \eqref{eqn-min-2}, we further have 
\begin{align}
&	\frac{\langle\boldsymbol{\xi},\bv\rangle}{\|\bv\|_{\V\cap\L^{r+1}}}\nonumber\\&\geq\frac{\left(\mu-m\lambda_0^{-1}\right)\|\bu\|_{\V}^2+\alpha\|\bv\|_{\H}^2+\frac{\beta}{2}\|\bv\|_{\L^{r+1}}^{r+1}-|\kappa|^{\frac{r+1}{r-q}}|\mathcal{O}|-m\lambda_0^{-1/2}\|\bv\|_{\V}-\|\Psi\|_{\V^*+\L^{\frac{r+1}{r}}}\|\v\|_{\V\cap\L^{r+1}}}{\sqrt{\|\bv\|_{\V}^2+\|\bv\|_{\L^{r+1}}^{2}}} \nonumber\\&\geq\frac{\min\left\{\left(\mu-m\lambda_0^{-1}\right),\frac{\beta}{2}\right\}\left(\|\bv\|_{\V}^2+\|\bv\|_{\L^{r+1}}^2-1\right)-|\kappa|^{\frac{r+1}{r-q}}|\mathcal{O}|}{\sqrt{\|\bv\|_{\V}^2+\|\bv\|_{\L^{r+1}}^{2}}}-m\lambda_0^{-1/2}-\|\Psi\|_{\V^*+\L^{\frac{r+1}{r}}},
\end{align} 
	where we have used the fact that $x^2\leq x^{r+1} + 1,$ for all $x\geq 0$ and $r\geq 1$. 
Therefore for $\mu> m\lambda_0^{-1}$, we immediately deduce 
\begin{align*}
	\lim\limits_{\|\bv\|_{\V\cap\L^{r+1}}\to\infty}	\frac{\langle\boldsymbol{\xi},\bv\rangle}{\|\bv\|_{\V\cap\L^{r+1}}}=\infty,
\end{align*}
so that the operator $\partial\mathcal{E}:\V\cap\wi{\L}^{r+1}\to 2^{\V^{*}+\L^{\frac{r+1}{r}}}$ is coercive. 
	\vskip 0.2cm
\noindent\textbf{Step 3:} \emph{Strongly convex type property of $\mathcal{E}:\V\cap\wi{\L}^{r+1}\to\R$.}
For any $\bv_1,\bv_2\in\V\cap\wi{\L}^{r+1}$ and $\boldsymbol{\xi}_i\in\partial\mathcal{E}(\bv_i)$, $i=1,2,$ we write 
\begin{align}
	\boldsymbol{\xi}_i= \A\bv_i+\alpha\bv_i+\beta\mathcal{C}(\bv_i)+\kappa\mathcal{C}_0(\bv_i)+\boldsymbol{\eta}_i-\Psi, \ \boldsymbol{\eta}_i\in\partial J(\bv_i), i=1,2. 
\end{align}
Let us now consider 
\begin{align}\label{eqn-min-1}
&	\langle\boldsymbol{\xi}_1-\boldsymbol{\xi}_2,\bv_1-\bv_2\rangle\nonumber\\&=\mu a(\bv_1-\bv_2,\bv_1-\bv_2)+\alpha a_0(\bv_1-\bv_2,\bv_1-\bv_2)+\beta[c(\bv_1,\bv_1-\bv_2)-c(\bv_2,\bv_1-\bv_2)]\nonumber\\&\quad+\kappa[c_0(\bv_1,\bv_1-\bv_2)-c_0(\bv_2,\bv_1-\bv_2)]+\langle\boldsymbol{\eta}_1-\boldsymbol{\eta}_2,\bv_1-\bv_2\rangle\nonumber\\&=\mu\|\bv_1-\bv_2\|_{\V}^2+\alpha\|\bv_1-\bv_2\|_{\H}^2+\beta[c(\bv_1-\bv_2,\bv_1-\bv_2)]+\langle\boldsymbol{\eta}_1-\boldsymbol{\eta}_2,\bv_1-\bv_2\rangle. 
\end{align}
Applying  \eqref{2.23}, we obtain 
\begin{align}\label{eqn-exist-1}
c(\bv_1,\bv_1-\bv_2)-c(\bv_2,\bv_1-\bv_2)\geq \frac{1}{2}\||\bv_1|^{\frac{r-1}{2}}(\bv_1-\v_2)\|_{\H}^2+\frac{1}{2}\||\v_2|^{\frac{r-1}{2}}(\bv_1-\v_2)\|_{\H}^2.
\end{align}
	An application of  Taylor's formula (\cite[Theorem 7.9.1]{PGC}) yields 
\begin{align}\label{eqn-est-c0-1}
&|\kappa[c_0(\bv_1,\bv_1-\bv_2)-c_0(\bv_2,\bv_1-\bv_2)]|\nonumber\\&=	|	\kappa||\langle\mathcal{C}_0(\bv_1)-\mathcal{C}_0(\bv_2),\bv_1-\bv_2\rangle|\nonumber\\&= |\kappa|\bigg|\bigg<\int_0^1\mathcal{C}_0^{\prime}(\theta\bv_1+(1-\theta)\bv_2)\d\theta(\bv_1-\bv_2),(\bv_1-\bv_2)\bigg>\bigg|\nonumber\\&\leq |\kappa| q2^{q-1}\bigg<\int_0^1|\theta\bv_1+(1-\theta)\bv_2|^{q-1}\d\theta|\bv_1-\bv_2|,|\bv_1-\bv_2|\bigg>\nonumber\\&\leq |\kappa|q2^{q-1}\left<\left(|\bv_1|^{q-1}+|\bv_2|^{q-1}\right)|\bv_1-\bv_2|,|\bv_1-\bv_2|\right>\nonumber\\&=   |\kappa| q2^{q-1}\||\bv_1|^{\frac{q-1}{2}}(\bv_1-\bv_2)\|_{\H}^2+ |\kappa| q2^{q-1}\||\bv_2|^{\frac{q-1}{2}}(\bv_1-\bv_2)\|_{\H}^2. 
\end{align}
Using H\"older's inequality, we estimate $ |\kappa| q2^{q-1}\||\bv_1|^{\frac{q-1}{2}}(\bv_1-\bv_2)\|_{\H}^2$ as 
\begin{align}\label{eqn-est-c0-2}
	& |\kappa| q2^{q-1}\||\bv_1|^{\frac{q-1}{2}}(\bv_1-\bv_2)\|_{\H}^2\nonumber\\&= |\kappa| q2^{q-1}\int_{\mathcal{O}}|\bv_1(\x)|^{q-1}|\bv_1(\x)-\bv_2(\x)|^2\d\x \nonumber\\&= |\kappa| q2^{q-1}\int_{\mathcal{O}}|\bv_1(\x)|^{q-1}|\bv_1(\x)-\bv_2(\x)|^{\frac{2(q-1)}{r-1}}|\bv_1(\x)-\bv_2(\x)|^{\frac{2(r-q)}{r-1}}\d\x \nonumber\\&\leq |\kappa| q2^{q-1}\bigg(\int_{\mathcal{O}}|\bv_1(\x)|^{r-1}|\bv_1(\x)-\bv_2(\x)|^2\d\x\bigg)^{\frac{q-1}{r-1}}\bigg(\int_{\mathcal{O}}|\bv_1(\x)-\bv_2(\x)|^2\d\x\bigg)^{\frac{r-q}{r-1}}\nonumber\\&\leq\frac{\beta}{4}\int_{\mathcal{O}}|\bv_1(\x)|^{r-1}|\bv_1(\x)-\bv_2(\x)|^2\d\x+\frac{\varrho_{r}}{2}\int_{\mathcal{O}}|\bv_1(\x)-\bv_2(\x)|^2\d\x,
\end{align}
where $\varrho_r$ is defined in \eqref{eqn-rho-2}. 
A similar calculation implies 
\begin{align}\label{eqn-est-c0-3}
	& |\kappa| q2^{q-1}\||\bv_2|^{\frac{q-1}{2}}(\bv_1-\bv_2)\|_{\H}^2\nonumber\\&\leq\frac{\beta}{4}\int_{\mathcal{O}}|\bv_2(\x)|^{r-1}|\bv_1(\x)-\bv_2(\x)|^2\d\x+\frac{\varrho_{r}}{2}\int_{\mathcal{O}}|\bv_1(\x)-\bv_2(\x)|^2\d\x.
\end{align}
Using \eqref{eqn-est-c0-2} and \eqref{eqn-est-c0-3} in \eqref{eqn-est-c0-1}, we arrive at
\begin{align}\label{eqn-est-c0-4}
&	|\kappa[c_0(\bv_1,\bv_1-\bv_2)-c_0(\bv_2,\bv_1-\bv_2)]|\nonumber\\&\leq\frac{\beta}{4}\||\bv_1|^{\frac{r-1}{2}}(\bv_1-\bv_2)\|_{\H}^2+\frac{\beta}{4}\||\bv_2|^{\frac{r-1}{2}}(\bv_1-\bv_2)\|_{\H}^2+\varrho_r\|\bv_1-\bv_2\|_{\H}^2. 
\end{align}
An application of \eqref{eqn-j0-diff-1} yields
\begin{align}\label{eqn-exist-2}
	\langle\boldsymbol{\eta}_1-\boldsymbol{\eta}_2,\bv_1-\bv_2\rangle\geq -m\lambda_0^{-1}\|\bv_1-\bv_2\|_{\V}^2. 
\end{align}
Substituting the  estimates \eqref{eqn-exist-1}, \eqref{eqn-est-c0-4} and \eqref{eqn-exist-2} in \eqref{eqn-min-1}, we deduce 
\begin{align}\label{eqn-min-6}
	\langle\boldsymbol{\xi}_1-\boldsymbol{\xi}_2,\bv_1-\bv_2\rangle&\geq(\mu-m\lambda_0^{-1})\|\bv_1-\bv_2\|_{\V}^2+(\alpha-\varrho_r)\|\bv_1-\bv_2\|_{\H}^{2}\nonumber\\&\quad+\frac{1}{4}\||\bv_1|^{\frac{r-1}{2}}(\bv_1-\v_2)\|_{\H}^2+\frac{1}{4}\||\v_2|^{\frac{r-1}{2}}(\bv_1-\v_2)\|_{\H}^2
\end{align}
But we know that 
\begin{align}\label{eqn-exist-3}
	\|\v_1-\v_2\|_{\L^{r+1}}^{r+1}
	& = \int_{\mathcal{O}}|\bv_1(\x) - v(\x)|^{r-1} |\bv_1(\x) - \bv_2(\x)|^2 \d\x\nonumber\\
	&\leq 2^{r-2} \int_{\mathcal{O}}(|\bv_1(\x)|^{r-1} + |\bv_2(\x)|^{r-1})|\bv_1(\x) - \bv_2(\x)|^2 \d\x\nonumber\\
	& = 2^{r-2} \||\bv_1|^{\frac{r-1}{2}}(\bv_1-\bv_2)\|^2_{\H} + 2^{r-2} \||\bv_2|^{\frac{r-1}{2}}(\bv_1-\bv_2)\|^2_{\H}.
\end{align}
Using \eqref{eqn-exist-3} in \eqref{eqn-min-6}, we deduce 
\begin{align}
		\langle\boldsymbol{\xi}_1-\boldsymbol{\xi}_2,\bv_1-\bv_2\rangle&\geq (\mu-m\lambda_0^{-1})\|\bv_1-\bv_2\|_{\V}^2+(\alpha-\varrho_r)\|\bv_1-\bv_2\|_{\H}^{2}+\frac{\beta}{2^r}\|\bv_1-\bv_2\|_{\L^{r+1}}^{r+1}\nonumber\\&\geq \rho\left(\|\bv_1-\bv_2\|_{\V}^2+\|\bv_1-\bv_2\|_{\L^{r+1}}^{r+1}\right),
\end{align}
where $\rho=\min\left\{(\mu-m\lambda_0^{-1}),\frac{\beta}{2^{r}}\right\}$. Since  $\mu\lambda_0>m$ and $\alpha>\varrho_r,$ by Sobolev's embedding, we know that $\V\hookrightarrow \L^{r+1},$  for $2\leq r+1\leq\frac{2d}{d-2}$ ($1\leq r<\infty,$ for $d=2$),   it follows from Lemma \ref{lem-strong-convex} that the energy functional $\mathcal{E}$ is strongly convex on $\V$.

For $d=3$ and $r\in(5,\infty)$, a different approach is required. We adopt the ideas presented in \cite[Theorem 3.4]{LFSLSG} to carry out the proof. Let the condition \eqref{eqn-min-6} be satisfied. Our aim is to show that  for any $\v_1,\v_2\in\V\cap\wi{\L}^{r+1}$ and any $\boldsymbol{\xi}\in\partial\mathcal{E}(\v_2)$, there exists a constant $\vartheta>0$ such that 
\begin{align}\label{eqn-strict-con}
	\mathcal{E}(\v_1)-\mathcal{E}(\v_2)\geq\langle \boldsymbol{\xi}, \v_1-\v_2\rangle+\vartheta\left(\|\bv_1-\bv_2\|_{\V}^2+\|\bv_1-\bv_2\|_{\L^{r+1}}^{r+1}\right). 
\end{align}
Assume to the contrary that $\mathcal{E}$  does not satisfy the condition \eqref{eqn-strict-con}, then, for any $\vartheta>0$, there exist  $\v_1,\v_2\in\V\cap\wi{\L}^{r+1}$ with $\v_1\neq \v_2$ and $\boldsymbol{\xi}\in\partial\mathcal{E}(\bv_2)$ such that 
\begin{align}\label{eqn-min-7}
	\mathcal{E}(\v_1)-	\mathcal{E}(\v_2)<\langle \boldsymbol{\xi}, \v_1-\v_2\rangle+\vartheta\left(\|\bv_1-\bv_2\|_{\V}^2+\|\bv_1-\bv_2\|_{\L^{r+1}}^{r+1}\right). 
\end{align}
By the mean value theorem (Lemma \ref{lem-mvt}), it follows that there exist $\tau_0\in (0,1)$  and $\boldsymbol{\xi}_0\in\partial\mathcal{E}(\bv_0)$  such that
\begin{align}\label{eqn-min-8}
		\mathcal{E}(\v_1)-	\mathcal{E}(\v_2)\leq\langle \boldsymbol{\xi}_0,\v_1-\v_2\rangle,
\end{align}
where $\v_0=\tau_0\v_1+(1-\tau_0)\v_2$. Subtracting \eqref{eqn-min-7} from \eqref{eqn-min-8}, we infer 
\begin{align}\label{eqn-min-9}
	\langle \boldsymbol{\xi}_0-\boldsymbol{\xi},\v_1-\v_2\rangle <\vartheta\left(\|\bv_1-\bv_2\|_{\V}^2+\|\bv_1-\bv_2\|_{\L^{r+1}}^{r+1}\right). 
\end{align}
Using  \eqref{eqn-min-6} and \eqref{eqn-min-9}, we deduce 
\begin{align}\label{eqn-min-10}
&\rho\left(\tau_0^2\|\bv_1-\bv_2\|_{\V}^2+\tau_0^{r+1}\|\bv_1-\bv_2\|_{\L^{r+1}}^{r+1}\right)\nonumber\\&=\rho\left(\|\bv_0-\bv_2\|_{\V}^2+\|\bv_0-\bv_2\|_{\L^{r+1}}^{r+1}\right)\nonumber\\&\leq 
	\langle\boldsymbol{\xi}_0-\boldsymbol{\xi},\bv_0-\bv_2\rangle=\tau_0\langle\boldsymbol{\xi}_0-\boldsymbol{\xi},\bv_1-\bv_2\rangle\nonumber\\&<\vartheta\tau_0\left(\|\bv_1-\bv_2\|_{\V}^2+\|\bv_1-\bv_2\|_{\L^{r+1}}^{r+1}\right). 
	\end{align}
	Since $\tau_0\in(0,1)$, we know that $\tau_0^{r+1}\leq\tau_0^2$, then from \eqref{eqn-min-10}, we deduce 
	\begin{align}
		0<(\vartheta-\rho\tau_0^r)\left(\|\bv_1-\bv_2\|_{\V}^2+\|\bv_1-\bv_2\|_{\L^{r+1}}^{r+1}\right),
	\end{align}
that is,  $\vartheta>\rho\tau_0^r,$ which  contradicts the arbitrariness of $\vartheta$. Therefore,  the energy functional $\mathcal{E}$ satisfies \eqref{eqn-strict-con} on $\V\cap\wi{\L}^{r+1}$.


	\vskip 0.2cm
\noindent\textbf{Step 4:} \emph{Existence of a minimizer for $\mathcal{E}:\V\cap\wi{\L}^{r+1}\to\R$.}
Let us now show that $\mathcal{E}$  has a minimizer $\bu$ on $\V\cap\wi{\L}^{r+1}$. Since $	\mathcal{E}$ on $\V\cap\wi{\L}^{r+1}$ is coercive, we infer 
$$
\lim_{\|\bv\|_{\V\cap\L^{r+1} \to \infty}} \mathcal{E}(\bv) = +\infty,
$$
which implies that the sublevel sets $\{ \bv \in \V\cap\wi{\L}^{r+1} : \mathcal{E}(\bv) \leq c \}$ are nonempty, closed, and bounded. Since $\V\cap\wi{\L}^{r+1}$ is reflexive, every bounded sequence has a weakly convergent subsequence. Assume $\{ \bv_k \}_{k\in\N} \subset \V\cap\wi{\L}^{r+1}$ is a minimizing sequence for $\mathcal{E}$,  that is,
$$
\lim_{k \to \infty} \mathcal{E}(\bv_k) = \inf_{\bv \in \V\cap\wi{\L}^{r+1}} \mathcal{E}(\bv).
$$
By coercivity, $\{ \bv_k \}_{k\in\N}$ is bounded, so there exists a subsequence (still denoted $\{\bv_k\}_{k\in\N}$) and some $\widetilde{\bv} \in \V\cap\wi{\L}^{r+1}$ such that $\bv_k \xrightarrow{w} \widetilde{\bv}$  in $\V\cap\wi{\L}^{r+1}$. Since the embedding $\V\subset\H$ is compact, there exists a subsequence of $\{ \bv_k \}_{k\in\N} $ (still denoted by the same symbol) such that $\bv_k\to\widetilde{\bv}$ in $\H$. Since every weakly convergent sequence is bounded, there exists a constant $C$ independent of $k$ such that $\|\bv_k\|_{\V} +\|\bv_k\|_{\L^{r+1}}\leq C$.  An application of interpolation inequality yields 
\begin{align*}
	\|\bv_k-\widetilde{\bv}\|_{\L^{q+1}}^{q+1}&\leq \|\bv_k-\bv\|_{\L^{r+1}}^{\frac{(r+1)(q-1)}{r-1}}\|\bv_k-\bv\|_{\H}^{\frac{2(r-q)}{r-1}}\leq \left(C+\|\v\|_{\L^{r+1}}\right)^{\frac{(r+1)(q-1)}{r-1}}\|\bv_k-\bv\|_{\H}^{\frac{2(r-q)}{r-1}}\nonumber\\&\to 0\ \text{ as }\ k\to\infty. 
\end{align*}
Therefore, we immediately have 
\begin{align}\label{eqn-min-15}
	\lim_{k\to\infty}\|\bv_k\|_{\L^{q+1}}^{q+1}=\|\widetilde{\bv}\|_{\L^{q+1}}^{q+1}. 
\end{align}
Since norms are weakly lower semicontinuous, we further infer
\begin{align}\label{eqn-min-16}
\|\widetilde{\bv}\|_{\V}^2\leq 	\liminf_{k\to\infty}\|\bv_k\|_{\V}^2\ \text{ and }\ \|\widetilde{\bv}\|_{\L^{r+1}}^2\leq 	\liminf_{k\to\infty}\|\bv_k\|_{\L^{r+1}}^{r+1}. 
\end{align}
Since the embedding of $\V\hookrightarrow\L^2(\Gamma)$ is compact (\cite[Theorem 6.2, Chapter 2]{JNe}), we infer from \eqref{eqn-def-J} that 
\begin{align}\label{eqn-min-17}
	J(\v_k):=\int_{\Gamma}j(\x,v_{k,n}(\x))\d\S(\x)\to \int_{\Gamma}j(\x,v_{n}(\x))\d\S(\x)=J(\v)\ \text{ as } k\to\infty. 
	\end{align}
Combining the convergences \eqref{eqn-min-15}-\eqref{eqn-min-17}, we deduce 
\begin{align*}
\mathcal{E}(\widetilde{\bv}) \leq \liminf_{k \to \infty} \mathcal{E}(\bv_k) = \inf_{\bv \in \V\cap\wi{\L}^{r+1}} \mathcal{E}(\bv),
\end{align*}
which shows that $\widetilde{\bv}$ is a minimizer.

	\vskip 0.2cm
\noindent\textbf{Step 5:} \emph{Uniqueness of the minimizer for $\mathcal{E}:\V\cap\wi{\L}^{r+1}\to\R$.} Let us now show that the minimizer obtained above for $\mathcal{E}$ is unique. 
Suppose $\bv_1, \bv_2 \in \V\cap\wi{\L}^{r+1}$ are two minimizers of $\mathcal{E}$. 
Since \eqref{eqn-strict-con} holds for any $\v_1,\v_2\in\V\cap\wi{\L}^{r+1}$ and any  $\boldsymbol{\xi}\in\partial\mathcal{E}(\v_2)$, then for any $\lambda\in[0,1]$, and any  $\widetilde{\boldsymbol{\xi}}\in\partial\mathcal{E}(\lambda\v_1+(1-\lambda)\v_2),$ we find 
\begin{align*}
&	\lambda\mathcal{E}(\v_1)-\lambda\mathcal{E}(\lambda\v_1+(1-\lambda)\v_2)
\nonumber\\&\geq \lambda\langle \widetilde{\boldsymbol{\xi}}, \v_1-(\lambda\v_1+(1-\lambda)\v_2)\rangle \nonumber\\&\quad+\vartheta\lambda\left(\| \v_1-(\lambda\v_1+(1-\lambda)\v_2)\|_{\V}^2+\| \v_1-(\lambda\v_1+(1-\lambda)\v_2)\|_{\L^{r+1}}^{r+1}\right)\nonumber\\&=\lambda(1-\lambda) \langle \widetilde{\boldsymbol{\xi}},\v_1-\v_2\rangle +\vartheta\lambda(1-\lambda)^2\left(\|\v_1-\v_2\|_{\V}^2+(1-\lambda)^{r-1}\|\v_1-\v_2\|_{\L^{r+1}}^{r+1}\right). 
\end{align*}
A similar calculation yields
\begin{align*}
	&(1-\lambda)\mathcal{E}(\v_2)-(1-\lambda)\mathcal{E}(\lambda\v_1+(1-\lambda)\v_2)
	\nonumber\\&\geq-\lambda(1-\lambda)\langle \widetilde{\boldsymbol{\xi}},\v_1-\v_2\rangle+\vartheta\lambda^2(1-\lambda)\left(\|\v_1-\v_2\|_{\V}^2+\lambda^{r-1}\|\v_1-\v_2\|_{\L^{r+1}}^{r+1}\right). 
\end{align*}
Adding the above two relations, we deduce 
\begin{align}\label{eqn-strict-convex}
	&\lambda\mathcal{E}(\v_1)+(1-\lambda)\mathcal{E}(\v_2)-\mathcal{E}(\lambda\v_1+(1-\lambda)\v_2)\nonumber\\&\geq \vartheta\lambda(1-\lambda)\|\v_1-\v_2\|_{\V}^2+\vartheta\lambda(1-\lambda)[\lambda^r+(1-\lambda)^r]\|\v_1-\v_2\|_{\L^{r+1}}^{r+1}. 
\end{align}
But since $\bv_1$ and $\bv_2$ are both minimizers, $\mathcal{E}(\bv_1) = \mathcal{E}(\bv_2) = \inf \mathcal{E}$, so the inequality becomes
\begin{align*}
&\mathcal{E}((1 - \lambda)\bv_1 + \lambda \bv_2) \nonumber\\&\leq \inf \mathcal{E} -\left\{\vartheta\lambda(1-\lambda)\|\v_1-\v_2\|_{\V}^2+\vartheta\lambda(1-\lambda)[\lambda^r+(1-\lambda)^r]\|\v_1-\v_2\|_{\L^{r+1}}^{r+1}\right\}. 
\end{align*}
This contradicts the minimality of $\bv_1$ and $\bv_2$ unless $\|\bv_1 - \bv_2\|_\V = 0$ and $\|\v_1-\v_2\|_{\L^{r+1}}=0$, that is, $\bv_1 = \bv_2$.
The functional $\mathcal{E}$ admits a unique minimizer in $\V\cap\wi{\L}^{r+1}$. 

\vskip 0.2cm
\noindent\textbf{Step 6:} \emph{Existence and uniqueness of solutions to Problems \ref{prob-hemi-1} and  \ref{prob-min}.} 
Since ihe functional $\mathcal{E}$ admits a unique minimizer,   $\bu\in \V\cap\wi{\L}^{r+1}$ satisfies the relation 
\begin{align*}
	\mathcal{E}^0(\bu;\v)\geq 0\ \text{ for all }\ \v\in\V\cap\wi{\L}^{r+1}. 
\end{align*}
Using \eqref{eqn-subdiff-2}, we further have  for all $ \v\in\V\cap\wi{\L}^{r+1}$
\begin{align*}
	\mathcal{E}^0(\bu;\v)\leq \mu a(\bu,\v)+\alpha a_0(\bu,\v)+\beta c(\bu,\v)+\kappa c_0(\bu,\v)+J^0(\v)- 	\langle\boldsymbol{\Psi},\v\rangle.
\end{align*}
Combining with the above relation, we immediately deduce   for all $\v\in\V\cap\wi{\L}^{r+1},$ 
\begin{align}\label{eqn-min-12}
	 \mu a(\bu,\v)+\alpha a_0(\bu,\v)+\beta c(\bu,\v)+\kappa c_0(\bu,\v)+J^0(\bu;\v)\geq 	\langle\boldsymbol{\Psi},\v\rangle. 
\end{align}
so that $\bu\in\V\cap\wi{\L}^{r+1}$  is a solution of Problem \ref{prob-hemi-1}.

We now proceed to establish the uniqueness of the solution to Problem  \ref{prob-hemi-1}. If $\widetilde{\bu}\in\V\cap\wi{\L}^{r+1}$ is another solution of Problem  \ref{prob-hemi-1}, then  for all  $\v\in\V\cap\wi{\L}^{r+1}$, we have 
\begin{align}\label{eqn-min-13}
	 \mu a(\widetilde{\bu},\v)+\alpha a_0(\widetilde{\bu},\v)+\beta c(\widetilde{\bu},\v)+\kappa c_0(\widetilde{\bu},\v)+J^0(\widetilde{\bu};\v)\geq 	\langle\boldsymbol{\Psi},\v\rangle. 
\end{align}
Taking $\v=\widetilde{\bu}-\bu$ in \eqref{eqn-min-12} and $\v=\bu-\widetilde{\bu}$ in \eqref{eqn-min-13}, and then adding them together, we find 
\begin{align}\label{eqn-min-14}
	&\mu a(\widetilde{\bu}-\bu, \widetilde{\bu}-\bu)+a_0(\widetilde{\bu}-\bu,\widetilde{\bu}-\bu)+\beta[c(\widetilde{\bu},\widetilde{\bu}-\bu)-c(\bu,\widetilde{\bu}-\bu)]\nonumber\\&\quad+\kappa[c_0(\widetilde{\bu},\widetilde{\bu}-\bu)-c_0(\bu,\widetilde{\bu}-\bu)]\nonumber\\&\leq J^0(\widetilde{\bu};\bu-\widetilde{\bu})+J^0(\bu;\widetilde{\bu}-\bu)\nonumber\\&\leq m\lambda_0^{-1}\|\widetilde{\bu}-\bu\|_{\V}^2,
\end{align}
where we have used \eqref{eqn-j0-diff} also. Using \eqref{eqn-coercive}, \eqref{eqn-est-c0-4}, \eqref{eqn-exist-3} and \eqref{Eqn-mon-lip} in \eqref{eqn-min-14}, we finally arrive at 
\begin{align}
	(\mu-m\lambda_0^{-1})\|\widetilde{\bu}-\bu\|_{\V}^2+(\alpha-\varrho_r)\|\widetilde{\bu}-\bu\|_{\H}^2+\frac{\beta}{2^{r}}\|\widetilde{\bu}-\bu\|_{\L^{r+1}}^{r+1}\leq 0. 
\end{align}
By the condition \eqref{eqn-small}, we conclude that $\widetilde{\bu}=\bu$, thus ensuring the uniqueness of the solution to Problem \ref{prob-hemi-1}. Hence, a function $\bu\in\V\cap\wi{\L}^{r+1}$ solves Problem \ref{prob-hemi-1} if and only if it also solves Problem \ref{prob-min}; moreover, both problems have a unique solution.
\end{proof}

	\begin{remark}\label{rem-strong-convex}
		The condition \eqref{eqn-strict-convex} closely resembles the definition of strong convexity, differing only by the presence of the additional factor in the inequality.
	\end{remark}
	
	We now turn our attention to a variant of Problem \ref{prob-hemi-1}
		\begin{problem}\label{prob-hemi-2}
		Find $\bu\in\V$ such that 
		\begin{align}\label{eqn-alternative}
			&\mu a(\bu,\v)+\alpha a_0(\bu,\v)+\beta c(\bu,\v)+\kappa c_0(\bu,\v)+\int_{\Gamma}j^0(u_n;v_n)\d S\nonumber\\&\geq \langle\f,\v\rangle-b(\w,\w,\v), \ \text{ for all }\ \v\in\V\cap\wi{\L}^{r+1}. 
		\end{align}
	\end{problem}
	\begin{theorem}\label{thm-alternative}
	Assume Hypothesis \ref{hyp-new-j}, let $\f \in \V^*$, and suppose that condition \eqref{eqn-small} holds. Then, for any $\w \in \V\cap \L^{r+1}$, Problem \ref{prob-hemi-2} admits a unique solution $\bu \in \V \cap \L^{r+1}$, which is also the unique solution of Problems \ref{prob-hemi-1} and \ref{prob-min}.
	\end{theorem}
	\begin{proof}
		From Theorem \ref{thm-main-linear},  we know that under the given assumptions, Problem \ref{prob-hemi-1} admits a unique solution $\bu \in \V \cap \L^{r+1}$,  which also uniquely solves Problem \ref{prob-min}. Since
		\begin{align*}
			J^0(\bu;\v)\leq\int_{\Gamma}j^0(u_n;v_n)\d S,
		\end{align*}
		we observe that the solution $\bu$ satisfies inequality \eqref{eqn-alternative}. In other words, $\bu \in \V \cap \L^{r+1}$   is also a solution to Problem \ref{prob-hemi-2}. The uniqueness of the solution to Problem \ref{prob-hemi-2} can be established by an argument similar to that used in the proof of Theorem \ref{thm-main-linear}, and is therefore omitted. Consequently, the statement of Theorem \ref{thm-alternative} holds.
	\end{proof}
	
	\subsection{Existence and uniqueness of original problem} In this subsection, we establish the existence and uniqueness of the solution to Problem \ref{prob-hemi-var} by applying a Banach fixed point argument for $d\in\{2,3\}$ with  $r\in[1,\frac{d+2}{d-2}]$ ($1\leq r<\infty$ for $d=2$) and the Schauder fixed point theorem for $d=3$ with $r\in(5,\infty)$, building on the results from the previous subsection.
	
	For a given $\w\in\V\cap\wi{\L}^{r+1}$, let us first establish a bound for $\bu\in\V\cap\wi{\L}^{r+1}$, where $\bu$ is the solution of Problem \ref{eqn-alternative}. 
	
	\begin{lemma}
Under the assumptions of Theorem \ref{thm-alternative}, 	let  $\bu \in \V \cap \L^{r+1}$ be the   unique solution of 	Problem \ref{prob-hemi-2}. Then, we have the following energy estimate:
For $d\in\{2,3\}$ with $r\in[1,\frac{d+2}{d-2}]$ ($r\in[1,\infty)$ for $d=2$),
\begin{align}\label{eqn-bound-n}
		\|\bu\|_{\V}&\leq  \sqrt{\frac{2|\kappa|^{\frac{r+1}{r-q}}|\mathcal{O}|}{(\mu-m\lambda_0^{-1})}}+\frac{\sqrt{2}\left(c_0^2|\Gamma|\lambda_0^{-1}+\|\f\|_{\V^*}\right)}{(\mu-m\lambda_0^{-1})}+\frac{\sqrt{2}C_b\|\w\|_{\V}^2}{(\mu-m\lambda_0^{-1})}. 
\end{align}
For $d=3$ with $r\in(5,\infty)$,
\begin{align}\label{eqn-bound}
	\|\bu\|_{\V\cap\L^{r+1}}&\leq\Bigg\{ \sqrt{\frac{2|\kappa|^{\frac{r+1}{r-q}}|\mathcal{O}|}{(\mu-m\lambda_0^{-1})}}+\frac{\sqrt{2}\left(c_0^2|\Gamma|\lambda_0^{-1}+\|\f\|_{\V^*}\right)}{(\mu-m\lambda_0^{-1})}\nonumber\\&\quad+\sqrt{\frac{2}{\beta}}\bigg[|\kappa|^{\frac{1}{r-q}}|\mathcal{O}|^{\frac{1}{r+1}}+\frac{\left(c_0^2|\Gamma|\lambda_0^{-1}+\|\f\|_{\V^*}\right)^{\frac{2}{r+1}}}{(\mu-m\lambda_0^{-1})^{\frac{1}{r+1}}}\bigg]\Bigg\}\nonumber\\&\quad+\frac{\sqrt{2}C_b\|\w\|_{\V}^2}{(\mu-m\lambda_0^{-1})}+\sqrt{\frac{2}{\beta}}\frac{C_b^{\frac{2}{r+1}}\|\w\|_{\V}^{\frac{4}{r+1}}}{(\mu-m\lambda_0^{-1})^{\frac{1}{r+1}}}.
\end{align}
	\end{lemma}
	\begin{proof}
	Let us take $\v=-\bu$ in \eqref{eqn-alternative} to find 
	\begin{align*}
		&\mu a(\bu,\bu)+\alpha a_0(\bu,\bu)+\beta c(\bu,\bu)+\kappa c_0(\w,\bu)\leq \int_{\Gamma}j^0(u_n;-u_n)\d S+ \langle\f,\bu\rangle-b(\w,\w,\bu),
	\end{align*}
	so that 
	\begin{align}\label{eqn-bound-1}
		\mu\|\bu\|_{\V}^2+\alpha\|\bu\|_{\H}^2+\beta\|\bu\|_{\L^{r+1}}^{r+1}+\kappa\|\bu\|_{\L^{q+1}}^{q+1}\leq \int_{\Gamma}j^0(u_n;-u_n)\d S+ \langle\f,\bu\rangle-b(\w,\w,\bu).
	\end{align}
	Using Hypothesis \ref{hyp-new-j} (H.3), (H.4) and Young's inequality, we estimate the first term from the right hand side of the inequality \eqref{eqn-bound-1} as (cf. \eqref{eqn-min-3} also)
	\begin{align}\label{eqn-bound-2}
		 \int_{\Gamma}j^0(u_n;-u_n)\d S&= \int_{\Gamma}j^0(u_n;-u_n)\d S+ \int_{\Gamma}j^0(0;u_n)\d S- \int_{\Gamma}j^0(0;u_n)\d S\nonumber\\&\leq\int_{\Gamma}m|u_n|^2\d S+\int_{\Gamma}c_0|u_n|\d S\nonumber\\&\leq m\lambda_0^{-1}\|\bu\|_{\V}^2+c_0|\Gamma|^{1/2}\lambda_0^{-1/2}\|\bu\|_{\V}. 
	\end{align}
	Applying  \eqref{eqn-bbound}, we estimate $|b(\w,\w,\bu)|$  as
	\begin{align}\label{eqn-bound-3}
		|b(\w,\w,\bu)|\leq C_b\|\w\|_{\V}^2\|\bu\|_{\V},
	\end{align}
	An application of the Cauchy-Schwarz inequality  yield
	\begin{align}\label{eqn-bound-4}
		| \langle\f,\bu\rangle|&\leq\|\f\|_{\V^*}\|\bu\|_{\V}. 
	\end{align}
Making use of \eqref{eqn-pump-est}, \eqref{eqn-bound-2},  \eqref{eqn-bound-3} and \eqref{eqn-bound-4} in \eqref{eqn-bound-1}, we arrive at 
	\begin{align*}
			&(\mu-m\lambda_0^{-1})\|\bu\|_{\V}^2+\alpha\|\bu\|_{\H}^2+\frac{\beta}{2}\|\bu\|_{\L^{r+1}}^{r+1}\nonumber\\&\leq |\kappa|^{\frac{r+1}{r-q}}|\mathcal{O}|+\left(c_0|\Gamma|^{1/2}\lambda_0^{-1/2}+C_b\|\w\|_{\V}^2+\|\f\|_{\V^*}\right)\|\bu\|_{\V}\nonumber\\&\leq |\kappa|^{\frac{r+1}{r-q}}|\mathcal{O}|+\frac{ (\mu-m\lambda_0^{-1})}{2}\|\bu\|_{\V}^2+\frac{1}{2(\mu-m\lambda_0^{-1})}{\left(c_0|\Gamma|^{1/2}\lambda_0^{-1/2}+C_b\|\w\|_{\V}^2+\|\f\|_{\V^*}\right)^2}. 
	\end{align*}
	Therefore, the above relation reduces to 
	\begin{align}\label{eqn-bound-7}
		&\frac{(\mu-m\lambda_0^{-1})}{2}\|\bu\|_{\V}^2+\alpha\|\bu\|_{\H}^2+\frac{\beta}{2}\|\bu\|_{\L^{r+1}}^{r+1}\nonumber\\&\leq  |\kappa|^{\frac{r+1}{r-q}}|\mathcal{O}|+\frac{\left(c_0^2|\Gamma|\lambda_0^{-1}+\|\f\|_{\V^*}\right)^2}{(\mu-m\lambda_0^{-1})}+\frac{C_b^2\|\w\|_{\V}^4}{(\mu-m\lambda_0^{-1})}.
	\end{align}
	For $d\in\{2,3\}$ with $r\in[1,\frac{d+2}{d-2}]$ ($r\in[1,\infty)$ for $d=2$), we  know that $\V\hookrightarrow\L^{r+1}$, so that from \eqref{eqn-bound-7}, we deduce \eqref{eqn-bound-n}. 
	
For $d=3$ with $r\in(5,\infty)$, we need to compute in the following way: 
	\begin{align}\label{eqn-bound-5}
		\|\bu\|_{\V\cap\L^{r+1}}&=\sqrt{	\|\bu\|_{\V}^2+\|\bu\|_{\L^{r+1}}^2}\nonumber\\&\leq \left\{\frac{2}{(\mu-m\lambda_0^{-1})}\left[|\kappa|^{\frac{r+1}{r-q}}|\mathcal{O}|+\frac{\left(c_0^2|\Gamma|\lambda_0^{-1}+\|\f\|_{\V^*}\right)^2}{(\mu-m\lambda_0^{-1})}+\frac{C_b^2\|\w\|_{\V}^4}{(\mu-m\lambda_0^{-1})}\right]\right.\nonumber\\&\left.\quad+\frac{2}{\beta}\left[|\kappa|^{\frac{r+1}{r-q}}|\mathcal{O}|+\frac{\left(c_0^2|\Gamma|\lambda_0^{-1}+\|\f\|_{\V^*}\right)^2}{(\mu-m\lambda_0^{-1})}+\frac{C_b^2\|\w\|_{\V}^4}{(\mu-m\lambda_0^{-1})}\right]^{\frac{2}{r+1}}\right\}^{1/2}.
	\end{align}
By using the fact that 
\begin{align*}
	\mbox{$(a+b)^{\theta}\leq a^{\theta}+b^{\theta}$ for all $a,b\geq 0$ and $\theta\in(0,1)$, }
\end{align*}
one can establish the bound \eqref{eqn-bound} from the estimate \eqref{eqn-bound-5}. 
	\end{proof}
	
	Let us now introduce a bounded set on $\V\cap\wi{\L}^{r+1}$ by 
	\begin{align}
		\mathcal{K}_{\f}:=\left\{\v\in\V\cap\wi{\L}^{r+1}:\|\v\|_{\V\cap\wi{\L}^{r+1}}\leq \rho_{\f}\right\},
	\end{align}
	where 
	\begin{align}\label{eqn-rhof}
		\rho_{\f}: &=\left\{
		\begin{array}{ll}
	2\Bigg\{	\sqrt{\frac{2|\kappa|^{\frac{r+1}{r-q}}|\mathcal{O}|}{(\mu-m\lambda_0^{-1})}}+\frac{\sqrt{2}\left(c_0^2|\Gamma|\lambda_0^{-1}+\|\f\|_{\V^*}\right)}{(\mu-m\lambda_0^{-1})}\Bigg\},&\text{ for }\ r\in[1,\frac{d+2}{d-2}], \\	2 \Bigg\{ \sqrt{\frac{2|\kappa|^{\frac{r+1}{r-q}}|\mathcal{O}|}{(\mu-m\lambda_0^{-1})}}+\frac{\sqrt{2}\left(c_0^2|\Gamma|\lambda_0^{-1}+\|\f\|_{\V^*}\right)}{(\mu-m\lambda_0^{-1})}&\\ \quad+\sqrt{\frac{2}{\beta}}\bigg[|\kappa|^{\frac{1}{r-q}}|\mathcal{O}|^{\frac{1}{r+1}}+\frac{\left(c_0^2|\Gamma|\lambda_0^{-1}+\|\f\|_{\V^*}\right)^{\frac{2}{r+1}}}{(\mu-m\lambda_0^{-1})^{\frac{1}{r+1}}}\bigg]\Bigg\},& \text{ for }\  r>\frac{d+2}{d-2}. 
		\end{array}
		\right. 
	\end{align}
	The final condition in \eqref{eqn-rhof} means that $r\in(3,\infty)$ for $d=3$. 
	\begin{proposition}\label{prop-contra}
		Let  the assumptions of Theorem \ref{thm-alternative} be satisfied. Then under the following smallness condition 
		\begin{align}\label{eqn-cond-sol}
			\left\{
			\begin{array}{ll}
					0<2\sqrt{2}C_b\rho_{\f}\leq (\mu-m\lambda_0^{-1}), &\text{ for }\ r\in[1,\frac{d+2}{d-2}],\\
				0<	\frac{\sqrt{2}C_b\rho_{\f}}{(\mu-m\lambda_0^{-1})}+\sqrt{\frac{2}{\beta}}\frac{C_b^{\frac{2}{r+1}}\rho_{\f}^{\frac{3-r}{r+1}}}{(\mu-m\lambda_0^{-1})^{\frac{1}{r+1}}}\leq\frac{1}{2},&\text{ for }\ r>\frac{d+2}{d-2},
			\end{array}
			\right.
		\end{align}
		for any $\w\in 	\mathcal{K}_{\f}$, the solution $\bu$ of Problem  \ref{prob-hemi-2}  also belongs to the set 	$\mathcal{K}_{\f}$.
	\end{proposition}
	\begin{proof}
	For $d\in\{2,3\}$ with $r\in[1,\frac{d+2}{d-2}]$ ($r\in[1,\infty)$ for $d=2$),	and $\w\in \mathcal{K}_{\f}$, from the estimate \eqref{eqn-bound-1}, we deduce 
	\begin{align}
		\|\bu\|_{\V}\leq \frac{\rho_{\f}}{2}+\frac{\sqrt{2}C_b\rho_{\f}^2}{(\mu-m\lambda_0^{-1})}\leq \frac{\rho_{\f}}{2}+\frac{\rho_{\f}}{2}=\rho_{\f}, 
	\end{align}
		so that $\bu\in\mathcal{K}_{\f}$.

	For $d=3$ with $r\in(3,\infty)$,		 from the uniform bound \eqref{eqn-bound}, we infer 
		\begin{align}
\|\bu\|_{\V\cap\wi{\L}^{r+1}}\leq \frac{\rho_{\f}}{2}+\frac{\sqrt{2}C_b\rho_{\f}^2}{(\mu-m\lambda_0^{-1})}+\sqrt{\frac{2}{\beta}}\frac{C_b^{\frac{2}{r+1}}\rho_{\f}^{\frac{4}{r+1}}}{(\mu-m\lambda_0^{-1})^{\frac{1}{r+1}}}\leq \frac{\rho_{\f}}{2}+\frac{\rho_{\f}}{2}=\rho_{\f},
		\end{align}
		hence  $\bu\in\mathcal{K}_{\f}$.
	\end{proof}

	\begin{theorem}\label{thm-contra}
		Let  the assumptions of Theorem \ref{thm-alternative} and condition \ref{eqn-cond-sol} be satisfied. Then Problem \ref{prob-hemi-var} has a unique solution $\bu\in\V\cap\wi{\L}^{r+1}$.
	\end{theorem}
	\begin{proof}
		Let us define the operator $\Phi:\mathcal{K}_{\f}\to\mathcal{K}_{\f}$ by setting 
		\begin{align*}
			{\Phi}(\w)=\bu, 
		\end{align*}
		where $\bu\in\mathcal{K}_{\f}$
		is the unique solution of Problem \ref{prob-hemi-2}  corresponding to $\w\in\V\cap\wi{\L}^{r+1}$. We now aim to show that the operator $\Phi$ is a contraction on $\mathcal{K}_{\f}$. Let $\w_1,\w_2\in\mathcal{K}_{\f}$ be such that $\bu_1=\Phi(\w_1)$ and $\bu_2=\Phi(\w_2)$ are the unique solutions of Problem \ref{prob-hemi-2}  corresponding to $\w_1,\w_2,$ respectively. Then, for any $\v\in\V\cap\wi{\L}^{r+1}$, we have 
				\begin{align}\label{eqn-contra-1}
				\mu a(\bu_1-\bu_2,\v)+\alpha a_0(\bu_1-\bu_2,\v)&+\beta [c(\bu_1,\v)-c(\bu_2,\v)]+\kappa [c_0(\bu_1,\v)-c_0(\bu_2,\v)]\nonumber\\+\int_{\Gamma}[j^0(u_{1,n};v_n)+j^0(u_{2,n};-v_n)]\d S&\geq b(\w_2,\w_2,\v)-b(\w_1,\w_1,\v). 
		\end{align}
		Taking $\v=\bu_2-\bu_1$ in \eqref{eqn-contra-1}, we find
		\begin{align}\label{eqn-contra-2}
			&\mu a(\bu_1-\bu_2,\bu_1-\bu_2)+\alpha a_0(\bu_1-\bu_2,\bu_1-\bu_2) \nonumber\\ &\quad+\beta [c(\bu_1,\bu_1-\bu_2)-c(\bu_2,\bu_1-\bu_2)]+\kappa [c_0(\bu_1,\bu_1-\bu_2)-c_0(\bu_2,\bu_1-\bu_2)]\nonumber\\&\leq \int_{\Gamma}[j^0(u_{1,n};u_{2,n}-u_{1,n})+j^0(u_{2,n};u_{1,n}-u_{2,n})]\d S \nonumber\\&\quad+ b(\w_2,\w_2,\v)-b(\w_1,\w_1,\v). 
		\end{align}
		Using estimates similar to  \eqref{eqn-exist-1} and  \eqref{eqn-est-c0-4} in \eqref{eqn-contra-2}, we deduce 
		\begin{align}\label{eqn-contra-3}
		&\mu\|\bu_1-\bu_2\|_{\V}^2+(\alpha-\varrho_r)\|\bu_1-\bu_2\|_{\H}^2+\frac{\beta}{2^r}\|\bu_1-\bu_2\|_{\L^{r+1}}^{r+1}\nonumber\\&\leq\int_{\Gamma}[j^0(u_{1,n};u_{2,n}-u_{1,n})+j^0(u_{2,n};u_{1,n}-u_{2,n})]\d S\nonumber\\&\quad+ b(\w_2,\w_2,\bu_1-\bu_2)-b(\w_1,\w_1,\bu_1-\bu_2),
		\end{align}
		where $\varrho_r$ is defined in \eqref{eqn-rho-2}. From \eqref{eqn-j0-diff}, we infer 
		\begin{align}\label{eqn-contra-31}
			&\int_{\Gamma}[j^0(u_{1,n};u_{2,n}-u_{1,n})+j^0(u_{2,n};u_{1,n}-u_{2,n})]\d S\leq m\lambda_0^{-1}\|\bu_1-\bu_2\|_{\V}^2.
			\end{align}
		By the properties of trilinear operator $b(\cdot,\cdot,\cdot)$ and \eqref{eqn-bbound}, we get 
			\begin{align}\label{eqn-contra-32}
				&b(\w_2,\w_2,\bu_1-\bu_2)-b(\w_1,\w_1,\bu_1-\bu_2)\nonumber\\&= b(\w_2-\w_1,\w_2,\bu_1-\bu_2)-b(\w_1,\w_1-\w_2,\bu_1-\bu_2)\nonumber\\&\leq C_b\left(\|\w_1\|_{\V}+\|\w_2\|_{\V}\right)\|\w_1-\w_2\|_{\V}\|\bu_1-\bu_2\|_{\V}\nonumber\\&\leq\frac{(\mu-m\lambda_0^{-1})}{2}\|\bu_1-\bu_2\|_{\V}^2+\frac{C_b^2}{2(\mu-m\lambda_0^{-1})}\left(\|\w_1\|_{\V}+\|\w_2\|_{\V}\right)^2\|\w_1-\w_2\|_{\V}^2. 
			\end{align}
			Applying  the  estimates \eqref{eqn-contra-31} and \eqref{eqn-contra-32} in \eqref{eqn-contra-3}, we arrive at
			\begin{align}\label{eqn-contra-4}
				&\frac{(\mu-m\lambda_0^{-1})}{2}\|\bu_1-\bu_2\|_{\V}^2+(\alpha-\varrho_r)\|\bu_1-\bu_2\|_{\H}^2+\frac{\beta}{2^r}\|\bu_1-\bu_2\|_{\L^{r+1}}^{r+1}\nonumber\\&\leq \frac{C_b^2}{2(\mu-m\lambda_0^{-1})}\left(\|\w_1\|_{\V}+\|\w_2\|_{\V}\right)^2\|\w_1-\w_2\|_{\V}^2. 
			\end{align}
	We divide the rest of the proof into the following two cases:

			\vskip 0.2cm
			\noindent 	\textbf{Case 1:} \emph{$d\in\{2,3\}$ with $1\leq r\leq\frac{d+2}{d-2}$ ($1\leq r<\infty$ for $d=2$).} 
	For $m<\mu\lambda_0  \text{ and } \alpha>\varrho_r,$	we first consider the case  $d\in\{2,3\}$ with $1\leq r\leq\frac{d+2}{d-2}$. 	By Sobolev's embedding, we have  $\V\hookrightarrow \L^{r+1},$  for $2\leq r+1\leq\frac{2d}{d-2}$ ($1\leq r<\infty,$ for $d=2$), which ensures that the norm in $\L^{r+1}$
			is controlled by the norm in $\V$. As a result, the ball $\mathcal{K}_{\f}$	can be equivalently defined in terms of the $\V$-norm. Since $\w_1,\w_2\in\mathcal{K}_f$,  from the above relation, we immediately have 
			\begin{align}\label{eqn-contra-5}
			\|\Phi(\w_1)-\Phi(\w_2)\|_{\V}&=	\|\bu_1-\bu_2\|_{\V}\leq\frac{\sqrt{2}C_b\rho_{\f}}{(\mu-m\lambda_0^{-1})}\|\w_1-\w_2\|_{\V}=\sigma_f\|\w_1-\w_2\|_{\V},
			\end{align}
			where 
			\begin{align}\label{eqn-contra-5-1}
			\sigma_{\f}=\frac{\sqrt{2}C_b\rho_{\f}}{(\mu-m\lambda_0^{-1})}.
		\end{align}
		Choosing 
			$\sigma_{\f}<1$, from \eqref{eqn-contra-5}, we conclude that the operator $\Phi:\mathcal{K}_{\f}\to\mathcal{K}_{\f}$  is a contraction. Note that the first condition in \eqref{eqn-cond-sol} ensures that $\sigma_{\f}<1$. Applying the Banach fixed point theorem (Theorem \ref{thm-BFT}), we deduce that there exists a unique $\bu^*\in\mathcal{K}_{\f}$ such that $\Phi(\bu^*)=\bu^*$.  By the definition of $\Phi(\cdot)$, $\bu^*\in\mathcal{K}_{\f}$ satisfies
				\begin{align}\label{eqn-contra-6}
				&\mu a(\bu,\v)+\alpha a_0(\bu,\v)+b(\bu,\bu,\v)+\beta c(\bu,\v)+\kappa c_0(\bu,\v)+\int_{\Gamma}j^0(u_n;v_n)\d S\nonumber\\&\geq \langle\f,\v\rangle, \ \text{ for all }\ \v\in\V\cap\wi{\L}^{r+1}, 
			\end{align}
			so that $\bu^*\in\mathcal{K}_{\f}$ is a solution to Problem \ref{prob-hemi-var}. 
			
		\vskip 0.2cm
	\noindent 	\textbf{Case 2:} \emph{$d=3$ with $r\in(5,\infty)$.} 	In the case  $d=3$ with $r\in(3,\infty)$, the estimate \eqref{eqn-contra-4} does not provide that $\Phi(\cdot)$ is a contraction as we have to work with the $\V\cap\wi{\L}^{r+1}$ norm. 
				Therefore, we apply Schauder’s fixed point theorem to obtain the existence of a fixed point. We show that the mapping $\Phi:\mathcal{K}_{\f}\to\mathcal{K}_{\f}$ is continuous and compact. For $m<\mu\lambda_0  \text{ and } \alpha>\varrho_r,$ from the estimate \eqref{eqn-contra-4}, we infer 
		\begin{align*}
		&	\|\bu_1-\bu_2\|_{\V}+\|\bu_1-\bu_2\|_{\L^{r+1}}\nonumber\\&\leq\frac{\sqrt{2}C_b\rho_{\f}}{(\mu-m\lambda_0^{-1})}\|\w_1-\w_2\|_{\V}+\left(\frac{2^rC_b^2\rho_{\f}^2}{(\mu-m\lambda_0^{-1})}\right)^{\frac{1}{r+1}}\|\w_1-\w_2\|_{\L^{r+1}}^{\frac{2}{r+1}}. 
		\end{align*}
		Therefore, we have 
		\begin{align}\label{eqn-contra-7}
			\|\Phi(\w_1)-\Phi(\w_2)\|_{\V\cap\wi{\L}^{r+1}}\leq\max\bigg\{\frac{\sqrt{2}C_b\rho_{\f}}{(\mu-m\lambda_0^{-1})},\left(\frac{2^rC_b^2\rho_{\f}^2}{(\mu-m\lambda_0^{-1})}\right)^{\frac{1}{r+1}}\bigg\}\|\w_1-\w_2\|_{\V\cap\wi{\L}^{r+1}}^{\frac{2}{r+1}}.
		\end{align}
		The continuity of the mapping follows immediately from the estimate \eqref{eqn-contra-7}. Indeed, if $\{\w_n\}_{n\in\N}$ is a sequence in $ \mathcal{K}_{\f}$ such that $\|\w_n-\w\|_{\V\cap\wi{\L}^{r+1}}\to 0$ in $\mathcal{K}_{\f}$, then, from \eqref{eqn-contra-7}, we infer  $	\|\Phi(\w_n)-\Phi(\w)\|_{\V\cap\wi{\L}^{r+1}}\to 0$, so that the mapping $\Phi:\mathcal{K}_{\f}\to\mathcal{K}_{\f}$ is continuous.

		 For $r>3$, we estimate $b(\w_1,\w_1-\w_2,\bu_1-\bu_2)$ using H\"older's and Young's inequalities as 
		 \begin{align}\label{eqn-contra-8}
		 &	|b(\w_1,\w_1-\w_2,\bu_1-\bu_2)|\nonumber\\&\leq\|\w_1\|_{\V}\|\w_1-\w_2\|_{\L^{\frac{2(r+1)}{r-1}}}\|\bu_1-\bu_2\|_{\L^{r+1}}\nonumber\\&\leq\|\w_1\|_{\V}\|\w_1-\w_2\|_{\H}^{\frac{r-3}{r-1}}\|\w_1-\w_2\|_{\L^{r+1}}^{\frac{2}{r-1}}\|\bu_1-\bu_2\|_{\L^{r+1}}\nonumber\\&\leq\frac{\beta}{2^{r+1}}\|\bu_1-\bu_2\|_{\L^{r+1}}^{r+1}+\left(\frac{r}{r+1}\right)\left(\frac{2^{r+1}}{\beta(r+1)}\right)^{\frac{1}{r}}\|\w_1-\w_2\|_{\H}^{\frac{(r+1)(r-3)}{r(r-1)}}\|\w_1-\w_2\|_{\L^{r+1}}^{\frac{2(r+1)}{r(r-1)}}\nonumber\\&\leq \frac{\beta}{2^{r+1}}\|\bu_1-\bu_2\|_{\L^{r+1}}^{r+1}+\rho_{\beta,r}\|\w_1-\w_2\|_{\H}^{\frac{(r+1)(r-3)}{r(r-1)}}\|\w_1-\w_2\|_{\L^{r+1}}^{\frac{2(r+1)}{r(r-1)}}
		 \nonumber\\&\leq \frac{\beta}{2^{r+1}}\|\bu_1-\bu_2\|_{\L^{r+1}}^{r+1}+\rho_{\beta,r}\left(\|\w_1\|_{\L^{r+1}}^{\frac{2(r+1)}{r(r-1)}}+\|\w_2\|_{\L^{r+1}}^{\frac{2(r+1)}{r(r-1)}}\right)\|\w_1-\w_2\|_{\H}^{\frac{(r+1)(r-3)}{r(r-1)}},
		 \end{align}
		since $\frac{2(r+1)}{r(r-1)}<1$, where $\rho_{\beta,r}=\left(\frac{r}{r+1}\right)\left(\frac{2^{r+1}}{\beta(r+1)}\right)^{\frac{1}{r}}$.

			Let us now estimate $b(\w_2-\w_1,\w_2,\bu_1-\bu_2)$ using \eqref{eqn-b-int}, H\"older's Gagliardo-Nirenberg's and Young's inequalities as 
		\begin{align}\label{eqn-contra-9}
		&	|b(\w_2-\w_1,\w_2,\bu_1-\bu_2)|\nonumber\\&=|(\w_2-\w_1,((\bu_1-\bu_2)\cdot\nabla)\w_2-(\w_2\cdot\nabla)(\bu_1-\bu_2))|\nonumber\\&\leq\|\nabla\w_2\|_{\H}\|\w_2-\w_1\|_{\L^3}\|\bu_1-\bu_2\|_{\L^6}+\|\nabla(\bu_1-\bu_2)\|_{\H}\|\w_2-\w_1\|_{\L^3}\|\w_2\|_{\L^6}\nonumber\\&\leq C\|\w_2\|_{\V}\|\w_1-\w_2\|_{\H}^{1/2}\|\w_1-\w_2\|_{\V}^{1/2}\|\bu_1-\bu_2\|_{\V}\nonumber\\&\leq 
		\frac{(\mu-m\lambda_0^{-1})}{2}\|\bu_1-\bu_2\|_{\V}^2+\frac{C}{(\mu-m\lambda_0^{-1})}\|\w_2\|_{\V}^2\|\w_1-\w_2\|_{\H}\|\w_1-\w_2\|_{\V}\nonumber\\&\leq C\|\w_2\|_{\V}\|\w_1-\w_2\|_{\H}^{1/2}\|\w_1-\w_2\|_{\V}^{1/2}\|\bu_1-\bu_2\|_{\V}\nonumber\\&\leq 
		\frac{(\mu-m\lambda_0^{-1})}{2}\|\bu_1-\bu_2\|_{\V}^2+\frac{C}{(\mu-m\lambda_0^{-1})}\|\w_2\|_{\V}^2(\|\w_1\|_{\V}+\|\w_2\|_{\V})\|\w_1-\w_2\|_{\H}. 
		\end{align}
		Substituting \eqref{eqn-contra-31},   \eqref{eqn-contra-8} and   \eqref{eqn-contra-9} in \eqref{eqn-contra-3}, we obtain 
			\begin{align}\label{eqn-contra-10}
			&\frac{(\mu-m\lambda_0^{-1})}{2}\|\bu_1-\bu_2\|_{\V}^2+(\alpha-\varrho_r)\|\bu_1-\bu_2\|_{\H}^2+\frac{\beta}{2^{r+1}}\|\bu_1-\bu_2\|_{\L^{r+1}}^{r+1}\nonumber\\&\leq \rho_{\beta,r}\left(\|\w_1\|_{\L^{r+1}}^{\frac{2(r+1)}{r(r-1)}}+\|\w_2\|_{\L^{r+1}}^{\frac{2(r+1)}{r(r-1)}}\right)\|\w_1-\w_2\|_{\H}^{\frac{(r+1)(r-3)}{r(r-1)}}\nonumber\\&\quad+\frac{C}{(\mu-m\lambda_0^{-1})}\|\w_2\|_{\V}^2(\|\w_1\|_{\V}+\|\w_2\|_{\V})\|\w_1-\w_2\|_{\H}. 
		\end{align}
		 Our next aim is to show that $\Phi$ is compact. Let $\{\w_n\}_{n\in\N}$ be a bounded sequence in $\mathcal{K}_{\f}$. We need to show that $\{\Phi(\w_n)\}_{n\in\N}$ has a convergent subsequence. Since $\{\w_n\}_{n\in\N}$  is a bounded sequence, we know that $\|\w_n\|_{\V\cap\wi{\L}^{r+1}}\leq\rho_{\f}$ and an application of the Banach-Alaoglu theorem yields the existence of a subsequence of $\{\w_n\}_{n\in\N}$  (still denoted by the same symbol) such that
		 \begin{align*}
		 	\w_n\xrightarrow{w}\w\ \text{ in }\ \V\cap\wi{\L}^{r+1}. 
		 \end{align*} 
		Due to the compactness of the embedding $\V\hookrightarrow\H$, we can extract a (not relabeled) subsequence that converges strongly in $\H$, that is, 
		 \begin{align*}
		 	\w_n\to \w\ \text{ strongly in }\ \H. 
		 \end{align*}
		 Therefore, 	for $m<\mu\lambda_0  \text{ and } \alpha>\varrho_r,$ from \eqref{eqn-contra-10}, we infer 
		 \begin{align}
		 &	\|\Phi(\w_n)-\Phi(\w)\|_{\V\cap\wi{\L}^{r+1}}\nonumber\\&\leq\left\{\frac{2}{(\mu-m\lambda_0^{-1})}\left(2\rho_{\beta,r}\rho_{\f}^{\frac{2(r+1)}{r(r-1)}}+\frac{C}{(\mu-m\lambda_0^{-1})}\rho_{\f}^3\right)\right\}^{1/2}\|\w_n-\w\|_{\H}^{\frac{(r+1)(r-3)}{2r(r-1)}}\nonumber\\&\quad+ \left\{\frac{2}{(\mu-m\lambda_0^{-1})}\left(2\rho_{\beta,r}\rho_{\f}^{\frac{2(r+1)}{r(r-1)}}+\frac{C}{(\mu-m\lambda_0^{-1})}\rho_{\f}^3\right)\right\}^{\frac{1}{r+1}}\|\w_n-\w\|_{\H}^{\frac{(r-3)}{r(r-1)}}\nonumber\\&\to 0\ \text{ as } \ n\to\infty,
		 \end{align}
		 so that the operator $\Phi:\mathcal{K}_{\f}\to\mathcal{K}_{\f}$ is compact. Therefore, by Schauder’s fixed point theorem $\Phi$ has at least one fixed point and is a solution to Problem \ref{prob-hemi-var}. 
		 \end{proof}

		 \begin{theorem}\label{thm-global}
		 	Let $\f\in\V^*$. Then 
		 	\begin{enumerate}
		 		\item for $d\in\{2,3\}$ with $r\in[1,3]$, 
		 		\begin{align}\label{eqn-rest-1}
		 	\mbox{$m<\mu\lambda_0$\ \text{ and }\  $\alpha>\varrho_{2,r}+\frac{C\upsilon_{\f}^{\frac{8}{4-d}}}{(\mu-m\lambda_0^{-1})^{\frac{4+d}{4-d}}},$ }
		 			\end{align}
		 			where 
		 			\begin{align*}
		 		\upsilon_{\f}:=	\sqrt{\frac{2|\kappa|^{\frac{r+1}{r-q}}|\mathcal{O}|}{(\mu-m\lambda_0^{-1})}}   +\frac{1}{(\mu-m\lambda_0^{-1})}{\left(c_0|\Gamma|^{1/2}\lambda_0^{-1/2}+\|\f\|_{\V^*}\right)}
		 			\end{align*}
		 			and 
		 				\begin{align*}
		 				\varrho_{2,r}=	2\left(\frac{r-q}{r-1}\right)\left(\frac{2(q-1)}{\beta(r-1)}\right)^{\frac{q-1}{r-q}}\left( |\kappa| q2^{q-1}\right)^{\frac{r-1}{r-q}},
		 			\end{align*}
		 		\item	for $d\in\{2,3\}$ with $r\in(3,\infty)$, 
		 		\begin{align}\label{eqn-rest-2}
		 \mbox{$m<\mu\lambda_0$ \ \text{ and }\ $\alpha>\varrho_{\frac{1}{2},r}+\widehat{\varrho}_{\frac{1}{2},r},$}
		 \end{align} 
		 where 
		 	\begin{align*}
		 	\widehat{\varrho}_{\frac{1}{2},r}=\left(	\frac{1}{(\mu-m\lambda_0^{-1})}\right)^{\frac{r-1}{r-3}}\left(\frac{r-3}{r-1}\right)\left(\frac{16}{\beta (r-1)}\right)^{\frac{2}{r-3}},
		 \end{align*}
		 and 
		 	\begin{align*}
		 	\varrho_{\frac{1}{2},r}=	2\left(\frac{r-q}{r-1}\right)\left(\frac{8(q-1)}{\beta(r-1)}\right)^{\frac{q-1}{r-q}}\left( |\kappa| q2^{q-1}\right)^{\frac{r-1}{r-q}},
		 \end{align*}
		 or
		 \begin{align*}
		 	\mbox{$m<\mu\lambda_0$, $\alpha>\varrho_{\theta,r}+\frac{1}{(\mu-m\lambda_0^{-1})}$ and $\beta>\frac{2}{(\mu-m\lambda_0^{-1})}$, }
		 \end{align*}
		 where 
		 \begin{align*}
		 		\varrho_{2,r}=	2\left(\frac{r-q}{r-1}\right)\left(\frac{4(q-1)}{\theta\beta(r-1)}\right)^{\frac{q-1}{r-q}}\left( |\kappa| q2^{q-1}\right)^{\frac{r-1}{r-q}}, 
		 \end{align*}
		 	\end{enumerate}
		 Problem \ref{prob-hemi-var} admits a unique solution $\bu\in\V\cap\wi{\L}^{r+1}$. 
		 
		 Moreover, for $d\in\{2,3\}$ with $1\leq r\leq\frac{d+2}{d-2}$ ($1\leq r<\infty$ for $d=2$), the mapping $ \V^*\ni \boldsymbol{f} \mapsto \boldsymbol{u} \in \V$ is Lipschitz continuous, and for $d=3$ with $r\in(3,\infty)$, the mapping $ \V^*\ni \boldsymbol{f} \mapsto \boldsymbol{u} \in \V\cap\wi{\L}^{r+1}$ is H\"older continuous. 
		 \end{theorem}
		 
		\begin{proof} 
		 We aim to show that the solution exists for $m<\mu\lambda_0$ and for any $\f\in\V^*$. 
		 \vskip 0.2 cm
		 \noindent
		 \textbf{Step 1:} \emph{A parameterized family.} 
		 For this purpose, we consider the following family of problems parameterized by $t\in[0,1]$: 
		 	\begin{align}\label{eqn-contra-6-1}
		 	&\mu a(\bu_t,\v)+\alpha a_0(\bu_t,\v)+b(\bu_t,\bu_t,\v)+\beta c(\bu_t,\v)+\kappa c_0(\bu_t,\v)+\int_{\Gamma}j^0(u_{t,n};v_n)\d S\nonumber\\&\geq \langle t\f,\v\rangle, \ \text{ for all }\ \v\in\V\cap\wi{\L}^{r+1}.
		 \end{align}
		 Let us define the solution map $[0,1]\ni t\mapsto\bu_t\in\V\cap\wi{\L}^{r+1}$. For $t=0$, the above problem becomes 
		 	\begin{align}\label{eqn-contra-6-2}
		 	&\mu a(\bu_0,\v)+\alpha a_0(\bu_0,\v)+b(\bu_0,\bu_0,\v)+\beta c(\bu_0,\v)+\kappa c_0(\bu_0,\v)+\int_{\Gamma}j^0(u_{0,n};v_n)\d S\geq 0,
		 \end{align}
		 which is solvable trivially if $\bu_0=\boldsymbol{0}$. By setting $t=1$, we obtain the original equation \eqref{eqn-contra-6}.

		  \vskip 0.2 cm
		 \noindent
		 \textbf{Step 2:} \emph{Uniform bounds.} 
	We first note that using similar bounds given in Proposition \ref{prop-contra} (see \eqref{eqn-bound-4-2} also),  by choosing $t$ appropriately and  using Theorem \ref{thm-contra}, we can show the existence of a maximal interval $[0,t_{\max})\subset[0,1]$ on which a solution exists for any $\f\in\V^*$.  	 Let us now show a uniform bound for any  solution  $\bu_t$ to the problem \eqref{eqn-contra-6-1}. Taking $\v=-\bu_t$ in \eqref{eqn-contra-6-1}, we find 
	\begin{align}
		&\mu\|\bu_t\|_{\V}^2+\alpha\|\bu_t\|_{\H}^2+\beta\|\bu_t\|_{\L^{r+1}}^{r+1}
		\nonumber\\&\leq-\kappa\|\bu_t\|_{\L^{q+1}}^{q+1}+\int_{\Gamma}j^0(u_{t,n};-u_{t,n})\d S+t \langle\f,\bu_t\rangle,
	\end{align}
	where we have used the fact that $b(\bu_t,\bu_t,\bu_t)=0$. 
	Applying  the Cauchy-Schwarz inequality, we get 
	\begin{align}\label{eqn-bound-4-1}
		|t \langle\f,\bu_t\rangle|&\leq t\|\f\|_{\V^*}\|\bu_t\|_{\V}\leq \|\f\|_{\V^*}\|\bu_t\|_{\V}. 
	\end{align}
	Using  \eqref{eqn-pump-est}, \eqref{eqn-bound-2} and \eqref{eqn-bound-4-1} in \eqref{eqn-bound-1}, we arrive at 
	\begin{align*}
		&(\mu-m\lambda_0^{-1})\|\bu_t\|_{\V}^2+\alpha\|\bu_t\|_{\H}^2+\frac{\beta}{2}\|\bu_t\|_{\L^{r+1}}^{r+1}\nonumber\\&\leq |\kappa|^{\frac{r+1}{r-q}}|\mathcal{O}|+\left(c_0|\Gamma|^{1/2}\lambda_0^{-1/2}+\|\f\|_{\V^*}\right)\|\bu_t\|_{\V}\nonumber\\&\leq |\kappa|^{\frac{r+1}{r-q}}|\mathcal{O}|+\frac{ (\mu-m\lambda_0^{-1})}{2}\|\bu_t\|_{\V}^2+\frac{1}{2(\mu-m\lambda_0^{-1})}{\left(c_0|\Gamma|^{1/2}\lambda_0^{-1/2}+\|\f\|_{\V^*}\right)^2},
	\end{align*}
	which implies 
	\begin{align}\label{eqn-bound-4-2}
		&(\mu-m\lambda_0^{-1})\|\bu_t\|_{\V}^2+\beta\|\bu_t\|_{\L^{r+1}}^{r+1}\leq  2|\kappa|^{\frac{r+1}{r-q}}|\mathcal{O}|+\frac{1}{(\mu-m\lambda_0^{-1})}{\left(c_0|\Gamma|^{1/2}\lambda_0^{-1/2}+\|\f\|_{\V^*}\right)^2}. 
	\end{align}
The above estimate provides a uniform bound on $\|\bu_t\|_{\V\cap\wi{\L}^{r+1}},$ $t\in[0,1]$, for $m<\mu\lambda_0$ and any $\f\in\V^*$.

		  \vskip 0.2 cm
		 \noindent
		 \textbf{Step 3:} \emph{Continuity of the map $t\mapsto\bu_t$.} 
		 We will now show that the mapping $[0,1]\ni t\mapsto\bu_t\in\V\cap\wi{\L}^{r+1}$ is continuous. For any $t_1,t_2\in[0,1]$ and $\bu_{t_1},\bu_{t_2}\in \V\cap\wi{\L}^{r+1}$, define  $\w_t=\bu_{t_1}-\bu_{t_2}$, which satisfies the following equation:  for all $\v\in\V\cap\wi{\L}^{r+1}, $
		 	\begin{align}\label{eqn-contra-6-3}
		 	&\mu a(\w_t,\v)+\alpha a_0(\w_t,\v)+b(\bu_{t_1},\bu_{t_1},\v)-b(\bu_{t_2},\bu_{t_2},\v)+\beta [c(\bu_{t_1},\v)-c(\bu_{t_2},\v)]\nonumber\\&+\kappa [c_0(\bu_{t_1},\v)-c_0(\bu_{t_2},\v)]+\int_{\Gamma}[j^0(u_{t_1,n};v_n)-j^0(u_{t_2,n};v_n)]\d S\geq \langle (t_1-t_2)\f,\v\rangle. 
		 \end{align}
		 	Taking $\v=\bu_{t_2}-\bu_{t_1}$ in \eqref{eqn-contra-6-3}, we get 
		 	\begin{align}\label{eqn-contra-6-4}
		 		&\mu\|\w_t\|_{\V}^2+\alpha\|\w_t\|_{\H}^2+\beta [c(\bu_{t_1},\w_t)-c(\bu_{t_2},\w_t)]
		 		\nonumber\\&\leq -b(\w_t,\bu_{t_2},\w_t) -\kappa [c_0(\bu_{t_1},\w_t)-c_0(\bu_{t_2},\w_t)]+\langle (t_1-t_2)\f,\w_t\rangle\nonumber\\&\quad+\int_{\Gamma}[j^0(u_{t_1,n};u_{t_2,n}-u_{t_1,n})+j^0(u_{t_2,n};u_{t_1,n}-u_{t_2,n})]\d S.
		 	\end{align}
		 	The term $c(\bu_{t_1},\w_t)-c(\bu_{t_2},\w_t)$ can be estimated using \eqref{2.23} as 
		 	\begin{align}\label{eqn-contra-6-5}
		 	c(\bu_{t_1},\w_t)-c(\bu_{t_2},\w_t)\geq 	 \frac{1}{2}\||\bu_{t_1}|^{\frac{r-1}{2}}\w_t\|_{\H}^2+\frac{1}{2}\||\bu_{t_2}|^{\frac{r-1}{2}}\w_t\|_{\H}^2. 
		 	\end{align}
		 		\vskip 0.2cm
		 	\noindent 	\textbf{Case 1:} \emph{$d\in\{2,3\}$ with $r\in(3,\infty)$.} 
		 	For $d\in\{2,3\}$ with $r\in(3,\infty)$, we estimate $b(\w_t,\bu_{t_2},\w_t)$ using \eqref{eqn-b-est}, H\"older's and Young's inequalities as 
		 	\begin{align}\label{eqn-contra-6-6}
		 		|b(\w_t,\bu_{t_2},\w_t)|&=|b(\w_t,\w_t,\bu_{t_2})|\leq\|\w_t\|_{\V}\|\bu_{t_2}\w_t\|_{\H}\nonumber\\&\leq\frac{(\mu-m\lambda_0^{-1})}{4}\|\w_t\|_{\V}^2+\frac{1}{(\mu-m\lambda_0^{-1})}\|\bu_{t_2}\w_t\|_{\H}^2. 
		 	\end{align}
		 	Once again using H\"older's and Young's inequalities, we estimate  $\frac{1}{(\mu-m\lambda_0^{-1})}\|\bu_{t_2}\w_t\|_{\H}^2$ as 
		 	\begin{align}\label{eqn-contra-6-7}
		 		\frac{1}{(\mu-m\lambda_0^{-1})}\|\bu_{t_2}\w_t\|_{\H}^2&=	\frac{1}{(\mu-m\lambda_0^{-1})}\int_{\mathcal{O}}|\bu_{t_2}(\x)|^2|\w_t(\x)|^2\d\x \nonumber\\&=\frac{1}{(\mu-m\lambda_0^{-1})}\int_{\mathcal{O}}|\bu_{2_{t}}(\x)|^2|\w_t(\x)|^{\frac{4}{r-1}}|\w_t(\x)|^{\frac{2(r-3)}{r-1}}\d \x\nonumber\\&\leq\frac{1}{(\mu-m\lambda_0^{-1})}\left(\int_{\mathcal{O}}|\bu_{2_{t}}(\x)|^{r-1}|\w_t(\x)|^2\d \x\right)^{\frac{2}{r-1}}\left(\int_{\mathcal{O}}|\w_t(\x)|^2\d \x\right)^{\frac{r-3}{r-1}}\nonumber\\&\leq\frac{\theta\beta}{4}\int_{\mathcal{O}}|\bu_{2_{t}}(\x)|^{r-1}|\w_t(\x)|^2\d \x+\widehat{\varrho}_{r}\int_{\mathcal{O}}|\w_t(\x)|^2\d \x,
		 	\end{align}
		 	where 
		 		\begin{align}\label{eqn-rho-1}
		 		\widehat{\varrho}_{\theta,r}=\left(	\frac{1}{(\mu-m\lambda_0^{-1})}\right)^{\frac{r-1}{r-3}}\left(\frac{r-3}{r-1}\right)\left(\frac{8}{\theta\beta (r-1)}\right)^{\frac{2}{r-3}},
		 	\end{align}
		 	for some $0<\theta\leq 1$. Substituting \eqref{eqn-contra-6-7} in \eqref{eqn-contra-6-6}, we obtain 
		 	\begin{align}\label{eqn-contra-6-8}
		 		&	|b(\w_t,\bu_{t_2},\w_t)|\nonumber\\&\leq \frac{(\mu-m\lambda_0^{-1})}{4}\|\w_t\|_{\V}^2+\frac{\theta\beta}{4}\int_{\mathcal{O}}|\bu_{2_{t}}(\x)|^{r-1}|\w_t(\x)|^2\d \x+\widehat{\varrho}_{r}\int_{\mathcal{O}}|\w_t(\x)|^2\d \x. 
		 	\end{align}
		 	A calculation similar to \eqref{eqn-est-c0-4} yields 
		 	\begin{align}\label{eqn-contra-6-9}
		 		|\kappa[c_0(\bu_{1_{t}},\w_t)-c_0(\bu_{2_{t}},\w_t)]|\leq\frac{\theta\beta}{4}\||\bu_{1_{t}}|^{\frac{r-1}{2}}(\w_t)\|_{\H}^2+\frac{\theta\beta}{4}\||\bu_{2_{t}}|^{\frac{r-1}{2}}(\w_t)\|_{\H}^2+\varrho_{\theta,r}\|\w_t\|_{\H}^2,
		 	\end{align}
		 for some $0<\theta\leq 2$,	where 
		 	\begin{align}\label{eqn-contra-6-10}
		 	\varrho_{\theta,r}=	2\left(\frac{r-q}{r-1}\right)\left(\frac{4(q-1)}{\theta\beta(r-1)}\right)^{\frac{q-1}{r-q}}\left( |\kappa| q2^{q-1}\right)^{\frac{r-1}{r-q}}. 
		 	\end{align}
		 	From \eqref{eqn-j0-diff}, we infer
		 	\begin{align}\label{eqn-contra-6-11}
		 		&\int_{\Gamma}[j^0(u_{t_1,n};u_{t_2,n}-u_{t_1,n})+j^0(u_{t_2,n};u_{t_1,n}-u_{t_2,n})]\d S\leq m\lambda_0^{-1}\|\w_t\|_{\V}^2.
		 	\end{align}
		 	Finally, we estimate $\langle (t_1-t_2)\f,\w_t\rangle$ using the Cauchy-Schwarz and Young's inequalities as 
		 	\begin{align}\label{eqn-contra-6-12}
		 		|\langle (t_1-t_2)\f,\w_t\rangle|&\leq|t_1-t_2|\|\f\|_{\V^*}\|\w_t\|_{\V}\nonumber\\&\leq\frac{(\mu-m\lambda_0^{-1})}{4}\|\w_t\|_{\V}^2+\frac{1}{(\mu-m\lambda_0^{-1})}(t_1-t_2)^2\|\f\|_{\V^*}^2. 
		 	\end{align}
		 	Combining \eqref{eqn-contra-6-5}, \eqref{eqn-contra-6-8}, \eqref{eqn-contra-6-9}, \eqref{eqn-contra-6-11} and \eqref{eqn-contra-6-12}, and substituting the resultant in \eqref{eqn-contra-6-4}, we arrive at 
		 		\begin{align}\label{eqn-contra-6-13}
		 		&\frac{(\mu-m\lambda_0^{-1})}{2}\|\w_t\|_{\V}^2+\left(\alpha-\varrho_{\theta,r}-\widehat{\varrho}_{\theta,r}\right)\|\w_t\|_{\H}^2\nonumber\\&\quad+\frac{\beta(2-\theta)}{2}\||\bu_{1_{t}}|^{\frac{r-1}{2}}(\w_t)\|_{\H}^2+\frac{\beta(1-\theta)}{4}\||\bu_{2_{t}}|^{\frac{r-1}{2}}(\w_t)\|_{\H}^2\nonumber\\&\leq\frac{1}{(\mu-m\lambda_0^{-1})}(t_1-t_2)^2\|\f\|_{\V^*}^2,
		 	\end{align}
		 for $0<\theta\leq 1$. 	By choosing $\theta=\frac{1}{2}$ in \eqref{eqn-contra-6-13}  and then using \eqref{eqn-exist-3} in \eqref{eqn-contra-6-13}, we deduce 
		 		\begin{align}\label{eqn-contra-6-14}
		 		&\frac{(\mu-m\lambda_0^{-1})}{2}\|\w_t\|_{\V}^2+\left(\alpha-\varrho_{\frac{1}{2},r}-\widehat{\varrho}_{\frac{1}{2},r}\right)\|\w_t\|_{\H}^2+\frac{\beta}{2^{r+1}}\|\w_t\|_{\L^{r+1}}^{r+1}\nonumber\\&\leq\frac{1}{(\mu-m\lambda_0^{-1})}(t_1-t_2)^2\|\f\|_{\V^*}^2,
		 	\end{align}
		 	and for $m<\mu\lambda_0$ and $\alpha>\varrho_{\frac{1}{2},r}+\widehat{\varrho}_{\frac{1}{2},r}.$   Therefore, the continuity of the mapping $t\mapsto\bu_t$  in $\V\cap\wi{\L}^{r+1}$ follows from \eqref{eqn-contra-6-14}.

		 		One can estimate $b(\w_t,\bu_{t_2},\w_t)$ in the following way also: 
		 	\begin{align}\label{2.26}
		 		&	|b(\w_t,\bu_{t_2},\w_t)|\nonumber\\&\leq\|\w_t\|_{\V}\|\bu_{t_2}\w\|_{\H}
		 		\nonumber\\&\leq \frac{(\mu-m\lambda_0^{-1})}{4} \|\w_t\|_{\V}^2+\frac{1}{(\mu-m\lambda_0^{-1}) }\int_{\mathcal{O}}|\bu_{t_2}(\x)|^2|\w_t(\x)|^2\d\x\nonumber\\&=  \frac{(\mu-m\lambda_0^{-1})}{4} \|\w_t\|_{\V}^2+\frac{1}{(\mu-m\lambda_0^{-1}) }\int_{\mathcal{O}}|\w_t(\x)|^2\left(|\bu_{t_2}(\x)|^{r-1}+1\right)\frac{|\bu_{t_2}(\x)|^2}{|\bu_{t_2}(\x)|^{r-1}+1}\d\x\nonumber\\&\leq \frac{(\mu-m\lambda_0^{-1})}{4} \|\w_t\|_{\V}^2+\frac{1}{(\mu-m\lambda_0^{-1}) }\int_{\mathcal{O}}|\bu_{t_2}(\x)|^{r-1}|\w_t(\x)|^2\d\x\nonumber\\&\quad+\frac{1}{(\mu-m\lambda_0^{-1}) }\int_{\mathcal{O}}|\w_t(\x)|^2\d\x,
		 	\end{align}
		 	where 	we have used the fact that $\left\|\frac{|\bu_{t_2}|^2}{|\bu_{t_2}|^{r-1}+1}\right\|_{\L^{\infty}}<1$, for $r\geq 3$. Therefore, \eqref{eqn-contra-6-4} reduces to 
		 		\begin{align}\label{eqn-contra-6-17}
		 		&\frac{(\mu-m\lambda_0^{-1})}{2}\|\w_t\|_{\V}^2+\left(\alpha-\varrho_{\theta,r}-\frac{1}{(\mu-m\lambda_0^{-1}) }\right)\|\w_t\|_{\H}^2\nonumber\\&\quad+\frac{\beta(2-\theta)}{2}\||\bu_{1_{t}}|^{\frac{r-1}{2}}(\w_t)\|_{\H}^2+\left(\frac{\beta}{2}-\frac{1}{(\mu-m\lambda_0^{-1}) }\right)\||\bu_{2_{t}}|^{\frac{r-1}{2}}(\w_t)\|_{\H}^2\nonumber\\&\leq\frac{1}{(\mu-m\lambda_0^{-1})}(t_1-t_2)^2\|\f\|_{\V^*}^2,
		 	\end{align}
for some $0<\theta<2$.	For  $m<\mu\lambda_0$, $\alpha>\varrho_{\theta,r}+\frac{1}{(\mu-m\lambda_0^{-1})}$ and $\beta>\frac{2}{(\mu-m\lambda_0^{-1})}$, we deduce 
	\begin{align}
			&\frac{(\mu-m\lambda_0^{-1})}{2}\|\w_t\|_{\V}^2+\left(\alpha-\varrho_{2,r}-\frac{1}{(\mu-m\lambda_0^{-1}) }\right)\|\w_t\|_{\H}^2\nonumber\\&\quad+\min\left\{\frac{\beta(2-\theta)}{2},\left(\frac{\beta}{2}-\frac{1}{(\mu-m\lambda_0^{-1}) }\right)\frac{1}{2^{r-1}}\right\}\|\w_t\|_{\L^{r+1}}^{r+1}\nonumber\\&\leq\frac{1}{(\mu-m\lambda_0^{-1})}(t_1-t_2)^2\|\f\|_{\V^*}^2.
	\end{align}
so that	 the continuity of the mapping $t\mapsto\bu_t$  in $\V\cap\wi{\L}^{r+1}$ follows. 

	\vskip 0.2cm
\noindent 	\textbf{Case 2:} \emph{$d\in\{2,3\}$ with $r\in[1,3]$.} 
Let us now consider the case $d\in\{2,3\}$ with $r\in[1,3]$.
From \eqref{eqn-bound-4-2}, we infer 
\begin{align}
	\|\bu_{t}\|_{\V}\leq\sqrt{\frac{2|\kappa|^{\frac{r+1}{r-q}}|\mathcal{O}|}{(\mu-m\lambda_0^{-1})}}   +\frac{1}{(\mu-m\lambda_0^{-1})}{\left(c_0|\Gamma|^{1/2}\lambda_0^{-1/2}+\|\f\|_{\V^*}\right)}=:\upsilon_{\f},
\end{align}
 In this case, we estimate  $|b(\w_t,\bu_{t_2},\w_t)|$ using H\"older's, Gagliardo-Nirenberg's and Young's inequalities, as 
\begin{align}\label{eqn-contra-6-15}
		|b(\w_t,\bu_{t_2},\w_t)|&\leq\|\w_t\|_{\V}\|\bu_{t_2}\|_{\L^4}\|\w_t\|_{\L^4}\leq C\|\w_t\|_{\V}^{1+\frac{d}{4}}\|\w_t\|_{\H}^{1-\frac{d}{4}}\|\bu_{t_2}\|_{\V}^{\frac{d}{4}}\|\bu_{t_2}\|_{\V}^{1-\frac{d}{4}}\nonumber\\&\leq \frac{(\mu-m\lambda_0^{-1})}{4}\|\w_t\|_{\V}^2+\frac{C}{(\mu-m\lambda_0^{-1})^{\frac{4+d}{4-d}}}\|\bu_{t_2}\|_{\V}^{\frac{8}{4-d}}\|\w_t\|_{\H}^2\nonumber\\&\leq \frac{(\mu-m\lambda_0^{-1})}{4}\|\w_t\|_{\V}^2+\frac{C\upsilon_{\f}^{\frac{8}{4-d}}}{(\mu-m\lambda_0^{-1})^{\frac{4+d}{4-d}}}\|\w_t\|_{\H}^2. 
\end{align}
Combining \eqref{eqn-contra-6-5}, \eqref{eqn-contra-6-15}, \eqref{eqn-contra-6-9}, \eqref{eqn-contra-6-11} and \eqref{eqn-contra-6-12}, and substituting the resultant in \eqref{eqn-contra-6-4}, we deduce 
\begin{align}\label{eqn-contra-6-16}
	&\frac{(\mu-m\lambda_0^{-1})}{2}\|\w_t\|_{\V}^2+\bigg(\alpha-\varrho_{2,r}-\frac{C\upsilon_{\f}^{\frac{8}{4-d}}}{(\mu-m\lambda_0^{-1})^{\frac{4+d}{4-d}}}\bigg)\|\w_t\|_{\H}^2\nonumber\\&\leq\frac{1}{(\mu-m\lambda_0^{-1})}(t_1-t_2)^2\|\f\|_{\V^*}^2. 
\end{align}
For $m<\mu\lambda_0$ and $\alpha>\varrho_{2,r}+\frac{C\upsilon_{\f}^{\frac{8}{4-d}}}{(\mu-m\lambda_0^{-1})^{\frac{4+d}{4-d}}},$   the continuity of the mapping $t\mapsto\bu_t$  in $\V\cap\wi{\L}^{r+1}$ follows. 

		   \vskip 0.2 cm
		 \noindent
		 \textbf{Step 4:} \emph{``Global'' solvability.} 
		Using Theorem \ref{thm-contra} and selecting $t$ appropriately, the bounds in Proposition \ref{prop-contra}  (see \eqref{eqn-bound-4-2} also) allow us to establish the existence of a maximal interval $[0, t_{\max}) \subset [0,1]$ on which a solution exists for any $\f \in \V^*$ and depends continuously on $t$.
		 As the solution $\bu_t$ remains uniformly bounded in $\V\cap\wi{\L}^{r+1}$ for $m<\mu\lambda_0, $  $\alpha>\varrho_{\frac{1}{2},r}+\widehat{\varrho}_{\frac{1}{2},r},$  and any $\f\in\V^*$ as $t\to t_{\max}$, then the solution can be extended to $t=t_{\max}$. Since a uniform bound exists for any $t\in[0,1]$, then we can extend the solution continuously to $t=1$. 
		 
		   \vskip 0.2 cm
		 \noindent
		 \textbf{Step 5:} \emph{Uniqueness and continuous dependence on the data.} 
	 For $m<\mu\lambda_0$ and $\alpha>\varrho_{\frac{1}{2},r}+\widehat{\varrho}_{\frac{1}{2},r},$	 the uniqueness of solutions to Problem \ref{prob-hemi-var} is immediate from \eqref{eqn-contra-6-14} by taking $t_1=t_2=1$. 
	 
	 It is now left to show Lipschitz continuity for $d\in\{2,3\}$ with $1\leq r\leq\frac{d+2}{d-2}$ ($1\leq r<\infty$ for $d=2$) and H\"older continuity for $d=3$ with $r\in(5,\infty)$. 
	 
	 Let $\bu_1$ and $\bu_2$ be two solutions of Problem \ref{prob-hemi-var} with the data $\f_1$ and $\f_2$, respectively. Then $\w=\bu_1-\bu_2$ satisfies the following: for all $\v\in\V\cap\wi{\L}^{r+1}, $
	 	\begin{align}\label{eqn-lip-1}
	 	&\mu a(\w,\v)+\alpha a_0(\w,\v)+b(\bu_{1},\bu_{1},\v)-b(\bu_{2},\bu_{2},\v)+\beta [c(\bu_{1},\v)-c(\bu_{2},\v)]\nonumber\\&+\kappa [c_0(\bu_{1},\v)-c_0(\bu_{2},\v)]+\int_{\Gamma}[j^0(u_{1,n};v_n)-j^0(u_{2,n};v_n)]\d S\geq \langle \f_1-\f_2,\v\rangle. 
	 \end{align}
	 Taking $\v=-\w$ in \eqref{eqn-lip-1}, we find 
	 \begin{align}\label{eqn-lip-2}
	 	&\mu\|\w\|_{\V}^2+\alpha\|\w\|_{\H}^2+\beta [c(\bu_{1},\w)-c(\bu_{2},\w)]
	 	\nonumber\\&\leq -b(\w,\bu_{2},\w) -\kappa [c_0(\bu_{1},\w)-c_0(\bu_{2},\w)]+\langle \f_1-\f_2,\w\rangle\nonumber\\&\quad+\int_{\Gamma}[j^0(u_{1,n};u_{2,n}-u_{1,n})+j^0(u_{2,n};u_{1,n}-u_{2,n})]\d S.
	 \end{align}
	 A calculation similar to \eqref{eqn-contra-6-14} yields 
		\begin{align}\label{eqn-lip-3}
		&\frac{(\mu-m\lambda_0^{-1})}{2}\|\w_t\|_{\V}^2+\left(\alpha-\varrho_{\frac{1}{2},r}-\widehat{\varrho}_{\frac{1}{2},r}\right)\|\w_t\|_{\H}^2+\frac{\beta}{2^{r+1}}\|\w_t\|_{\L^{r+1}}^{r+1}\nonumber\\&\leq\frac{1}{(\mu-m\lambda_0^{-1})}\|\f_1-\f_2\|_{\V^*}^2,
	\end{align}
	For $m<\mu\lambda_0$ and $\alpha>\varrho_{\frac{1}{2},r}+\widehat{\varrho}_{\frac{1}{2},r},$ the required properties hold. For the case $d\in\{2,3\}$ and $r\in[1,3]$, Lipschitz  continuity follows from \eqref{eqn-contra-6-16}. 
	\end{proof}
	
	\begin{remark}
 Although Step 4 in the proof of Theorem \ref{thm-global} relies on ``global'' solvability, we observe that part (1) of Theorem \ref{thm-global} imposes a constraint on the forcing term $\f$ (as seen in \eqref{eqn-rest-1}), whereas part (2) holds without any dependence on $\f$ (see \eqref{eqn-rest-2}). Up to a certain extent, Theorem \ref{thm-global} provides an improvement over \cite[Theorem 3.7]{MLWH} concerning the 2D and 3D stationary Navier-Stokes hemivariational inequality.
 \end{remark}

	\begin{remark}
 For the case $d\in\{2,3\}$ and $r=3$, one can estimate $	b(\w_t,\bu_{t_2},\w_t)$  in the following way: 
	\begin{align}\label{eqn-lip-4}
	&	|b(\w_t,\bu_{t_2},\w_t)|\leq\|\w_t\|_{\V}\|\bu_{t_2}\w_t\|_{\H}
\leq \frac{(\mu-m\lambda_0^{-1})}{4} \|\w_t\|_{\V}^2+\frac{1}{(\mu-m\lambda_0^{-1})}\|\bu_{t_2}\w_t\|_{\H}^2. 
		\end{align}
	Combining \eqref{eqn-contra-6-5}, \eqref{eqn-lip-4}, \eqref{eqn-contra-6-9}, \eqref{eqn-contra-6-11} and \eqref{eqn-contra-6-12}, and substituting the resultant in \eqref{eqn-contra-6-4}, we arrive at 
	\begin{align}\label{eqn-lip-5}
		&\frac{(\mu-m\lambda_0^{-1})}{2}\|\w_t\|_{\V}^2+\left(\alpha-\varrho_{\theta,3}\right)\|\w_t\|_{\H}^2\nonumber\\&\quad+\frac{\beta(2-\theta)}{2}\|\bu_{1_{t}}\w_t)\|_{\H}^2+\left(\frac{\beta}{2}-\frac{1}{(\mu-m\lambda_0^{-1})}\right)\|\bu_{2_{t}}\w_t\|_{\H}^2\nonumber\\&\leq\frac{1}{(\mu-m\lambda_0^{-1})}(t_1-t_2)^2\|\f\|_{\V^*}^2,
		\end{align}
		By choosing $\theta=2$, we derive from \eqref{eqn-lip-5} that 
		\begin{align}\label{eqn-lip-6}
			&\frac{(\mu-m\lambda_0^{-1})}{2}\|\w_t\|_{\V}^2+\left(\alpha-\varrho_{2,3}\right)\|\w_t\|_{\H}^2+\left(\frac{\beta}{2}-\frac{1}{(\mu-m\lambda_0^{-1})}\right)\|\bu_{2_{t}}\w_t\|_{\H}^2\nonumber\\&\leq \frac{1}{(\mu-m\lambda_0^{-1})}(t_1-t_2)^2\|\f\|_{\V^*}^2,
		\end{align}
	Note that $\varrho_{2,3}$ can be obtained by taking $\theta=2$ and $r=3$ in \eqref{eqn-contra-6-10}. 	Therefore, for $m<\mu\lambda_0$, $\alpha>\varrho_{2,3}$ and $\beta\geq \frac{2}{(\mu-m\lambda_0^{-1})}$, using the relation \eqref{eqn-lip-6}, one  can establish the continuity of the mapping $t\mapsto\bu_t$  in $\V\cap\wi{\L}^{r+1}$. 
	\end{remark}

	\subsection{A convergent iteration algorithm}
	Beyond guaranteeing the existence of a unique fixed point for a contractive mapping on a closed subset of a complete metric space, the Banach fixed point theorem also ensures the convergence of the fixed point iteration method. For $d\in\{2,3\}$ with $r\in[1,\frac{d+2}{d-2}]$ ($1\leq r<\infty$ for $d=2$), consider the following iterative scheme under the first  smallness condition \eqref{eqn-cond-sol} given in Proposition \ref{prop-contra}. 
	
	\begin{algorithm}\label{alg-num}
		\emph{Initialization:} Choose an initial guess $\bu_0\in\mathcal{K}_{\f},$ for  example $\bu_0=\boldsymbol{0}$. \\
		\emph{Iteration:} For $k\geq 1$, find $\bu_k\in\V$ such that 
			\begin{align}\label{eqn-algo}
			&\mu a(\bu_k,\v)+\alpha a_0(\bu_k,\v)+\beta c(\bu_k,\v)+\kappa c_0(\bu_k,\v)+\int_{\Gamma}j^0(u_{k,n};v_n)\d S\nonumber\\&\geq \langle\f,\v\rangle-b(\bu_{k-1},\bu_{k-1},\v), \ \text{ for all }\ \v\in\V.
		\end{align}
		\end{algorithm}
	Applying \cite[Theorem 5.1.3]{KAWH}, we have the following result: 
	\begin{theorem}\label{thm-num-alg}
	Under the first condition given in \eqref{eqn-cond-sol}, for $\f\in\V^*$, let $\bu\in\V$ be  the unique solution of Problem \ref{prob-hemi-var}. Then the sequence $\{\bu_k\}_{k\in\N}$ defined by \eqref{eqn-algo} converges to $\bu$, that is,
	\begin{align*}
		\|\bu_k-\bu\|_{\V}\to 0\ \text{ as } \ k\to\infty. 
	\end{align*}
	Moreover, the following error estimates hold:
	\begin{align}
		\|\bu_k-\bu\|_{\V}&\leq\sigma_{\f}\|\bu_{k-1}-\bu\|_{\V},\label{eqn-err-fir}\\
			\|\bu_k-\bu\|_{\V}&\leq\frac{\sigma_{\f}}{1-\sigma_{\f}}\|\bu_{k-1}-\bu_k\|_{\V},\label{eqn-err-sec}\\
				\|\bu_k-\bu\|_{\V}&\leq\frac{\sigma_{\f}^k}{1-\sigma_{\f}}\|\bu_{0}-\bu_1\|_{\V},\label{eqn-err-thir}
	\end{align}
	where $\sigma_{\f}$ is defined in \eqref{eqn-contra-5-1}.
	\end{theorem}

	\begin{remark}
		Under the smallness condition given in \eqref{eqn-cond-sol}, the constant $\sigma_{\f}$, defined in \eqref{eqn-contra-5-1}, satisfies $\sigma_{\f}< 1$. The first error estimate \eqref{eqn-err-fir} demonstrates that the iteration algorithm converges linearly, with convergence rate $\sigma_{\f}$. The second bound \eqref{eqn-err-sec} provides an a-posteriori error estimate, once the iterates $\bu_{k-1}$ and $\bu_k$ are computed, they can be used to obtain a computable upper bound on the error $\|\bu_k - \bu\|_{\V}$. The third bound \eqref{eqn-err-thir} gives an a-priori error estimate, given the initial guess $\bu_0$ and the first iterate $\bu_1$, it allows for estimating the number of iterations required to ensure that the error $\|\bu_k - \bu\|_{\V}$ falls within a prescribed tolerance.
	\end{remark}

	\begin{remark}
	When Problem \ref{prob-hemi-var} is solved using a numerical method, such as the finite element method, the discretized CBFeD hemivariational inequality can similarly be approximated by Algorithm \ref{alg-num} at the discrete level. Under the given assumptions, convergence of the discrete iteration is ensured, and the three corresponding error estimates remain valid for the discrete solutions.
	\end{remark}

	\begin{remark}
		Under the second condition given in \eqref{eqn-cond-sol}, for $\f\in\V^*$,  the first part of Theorem \ref{thm-num-alg} holds true for $d=3$ with $r\in(5,\infty)$. 
	\end{remark}
	
	\medskip\noindent
	\textbf{Acknowledgments:} Support for M. T. Mohan's research received from the National Board of Higher Mathematics (NBHM), Department of Atomic Energy, Government of India (Project No. 02011/13/2025/NBHM(R.P)/R\&D II/1137).

	\medskip\noindent	{\bf  Declarations:} 
	
	\noindent 	{\bf  Ethical Approval:}   Not applicable 
	
	\noindent  {\bf   Competing interests: } The author declare no competing interests.

	\noindent 	{\bf   Availability of data and materials: } Not applicable.

\end{document}